\documentclass[12pt]{amsart}
\usepackage{amsmath,amsfonts,amsthm,amscd,amssymb,mathrsfs}

\newtheorem{theorem}{Theorem}
\newtheorem{claim}[theorem]{Claim}
\newtheorem{proposition}[theorem]{Proposition}
\newtheorem{lemma}[theorem]{Lemma}
\newtheorem{definition}[theorem]{Definition}

\newtheorem{corollary}[theorem]{Corollary}
\newtheorem{question}[theorem]{Question}

\newtheorem{remark}[theorem]{Remark}
\newtheorem{example}[theorem]{Example}
\renewcommand{\wp}{warped product}
\newcommand{\wps}{warped product structure\ }
\newcommand{\tl}{\tilde{L}}
\def\bar{\overline}
\newcommand{\FF}{f}
\newcommand{\FFF}{F}
\newcommand{\SO}{\mathrm{SO}}
\newcommand{\Spin}{\mathrm{Spin}}
\newcommand{\SU}{\mathrm{SU}}
\newcommand{\U}{\mathrm{U}}
\newcommand{\ee}{e}
\newcommand{\vv}{v}
\newcommand{\wt}{\widetilde}
\newcommand{\Q}{Q}
\newcommand{\X}{X}
\newcommand{\Y}{Y}
\newcommand{\Z}[1]{Z^+_{#1}}
\newcommand{\geu}{g_E}
\newcommand{\Ext}{\raise1pt\hbox{$\bigwedge$}}
\newcommand{\Ann}{\mathop{\mathrm{Ann}}}
\newcommand{\Pf}{\mathop{\mathrm{Pf}}}
\newcommand{\Span}{\mathop{\mathrm{span}}}
\newcommand{\col}[1]{\left(\!\!\begin{array}{c}#1\end{array}\!\!\right)}
\newcommand{\mat}[1]{\left(\!\begin{array}{cc}#1\end{array}\!\right)}
\newcommand{\tp}{^{\kern-1pt\top}\kern-1pt}
\newcommand{\od}{\overline\partial}
\newcommand{\pd}{\partial}
\newcommand{\hs}{\\[3pt]\phantom{mmm}}
\newcommand{\suml}{\textstyle\sum\limits}
\newcommand{\bJ}{\kern1.5pt\mathbb{J}\kern.5pt}
\newcommand{\bfW}{\mathbf{W}}

\newcommand{\bfy}{\mathbf{y}}
\newcommand{\bfz}{\mathbf{z}}
\newcommand{\sE}{\mathscr{E}}

\newcommand{\sJ}{\mathscr{J}}
\newcommand{\sZ}{\mathscr{Z}}
\newcommand{\Mz}{M(\bfz)}
\newcommand{\const}{\mathrm{constant}}
\renewcommand{\ge}{\geqslant}
\renewcommand{\le}{\leqslant}
\newcommand{\frc}[2]{\hbox{\large$\frac{#1}{#2}$}}

\numberwithin{equation}{section}
\numberwithin{theorem}{section}

\newcommand{\pzbj}{\partial / \partial \bar{z}_j}
\newcommand{\pzi}{\partial / \partial z^i}

\newcommand{\bW}{\bar{W}}
\newcommand{\bz}{\bar{z}}

\newcommand{\cc}{\mathbb{C}}
\newcommand{\rr}{\mathbb{R}}
\newcommand{\cp}{\mathbb{P}}
\newcommand{\PP}{\mathbb{P}}
\renewcommand{\ge}{\geqslant}
\renewcommand{\le}{\leqslant}
\newcommand{\bxi}{\hbox{\boldmath$\xi$}}
\newcommand{\bet}{\hbox{\boldmath$\eta$}}
\newcommand{\bzt}{\hbox{\boldmath$\zeta$}}
\newcommand{\sbxi}{{\lower2pt\hbox{\scriptsize\boldmath$\xi$}}}
\begin{document}

\bibliographystyle{amsalpha} 
\title[Orthogonal complex structures]{Twistor geometry and warped product\\[3pt] 
orthogonal complex structures} 
\author{Lev Borisov}
\thanks{Research of the first author partially supported by NSF Grant DMS-0758480}
\address{Department of Mathematics, University of Wisconsin, Madison, 
WI, 53706}
\email{borisov@math.wisc.edu}
\author{Simon Salamon}
\thanks{Research of the second author partially supported by MIUR (Metriche Riemanniane e
   Variet\`a Differenziabili, PRIN\,05)}
\address{Simon Salamon, Dipartimento di Matematica, Politecnico di
  Torino, Corso Duca degli Abruzzi 24, 10129 Torino, Italy.}
\email{salamon@calvino.polito.it}
\author{Jeff Viaclovsky}
\address{Department of Mathematics, University of Wisconsin, Madison, 
WI, 53706}
\email{jeffv@math.wisc.edu}
\thanks{Research of the third author partially supported by NSF Grant DMS-0804042}

\begin{abstract} The twistor space of the sphere $S^{2n}$ is an
  isotropic Grassmannian that fibers over $S^{2n}$. An orthogonal
  complex structure on a subdomain of $S^{2n}$ (a complex
  structure compatible with the round metric) determines a section of
  this fibration with holomorphic image. In this paper, we use this
  correspondence to prove that any finite energy orthogonal complex
  structure on $\rr^6 \subset S^6$ must be of a special \wp\ form, and
  we also prove that any orthogonal complex structure on $\rr^{2n}$
  that is asymptotically constant must itself be constant. We will
  also give examples defined on $\rr^{2n}$ which have infinite energy,
  and examples of non-standard orthogonal complex structures on flat
  tori in complex dimension three and greater.
\end{abstract}

\date{May 22, 2009}
\maketitle
\setcounter{tocdepth}{1}
\vspace{-5mm}
\tableofcontents
\vspace{-10mm}
\parskip1pt

\section{Introduction}

An orthogonal complex structure on a Riemannian manifold $(M,g)$ is a
complex structure which is integrable and is compatible with the
metric $g$. In a previous paper of the second and third authors, the
case of domains in $\rr^4$ with the Euclidean metric was considered,
and various Liouville-type theorems were proved
\cite{SalamonViaclovsky}.  In particular, it was shown that if $J$ is
an orthogonal complex structure on $\mathbb{R}^4 \setminus K$, where
$K$ is a closed set with $\mathcal{H}^1(K) = 0$ (vanishing
$1$-dimensional Hausdorff measure), then $J$ is conformally equivalent
to a constant OCS. This generalized an earlier result of Wood
\cite{WoodIJM} and equivalent arguments of LeBrun--Poon
\cite{LeBrunPoon}.  

In this paper, we consider the case of higher-dimensional Euclidean
spaces, and then focus on the case of $\rr^6$. The 4-dimensional
Liouville theorem does not directly generalize to higher dimensions.
Indeed, the twistor construction over $\rr^4$ itself gives rise to an
orthogonal complex structure on $\rr^6$ that is not conformally
constant, see Remark~\ref{4dt} below. Global examples of OCSes 
on $\rr^{2n}$ for $n\ge 3$ which are not conformally constant were described
explicitly by Baird--Wood \cite{BairdWood1995} in the context of
harmonic morphisms (see \cite{Salamon1985,BairdWood,GudmundssonWood}
for other relevant links with harmonic maps and morphisms).

We next discuss a method for writing down many examples of global
OCSes on $\rr^6$. Endow $\rr^{6} = \cc^3$ with complex coordinates
$(z^1, z^2, z^3)$, and consider an orthogonal almost complex structure
of the form
\begin{align}
\label{wpoi}
J = J_1(z^3) \oplus J_0,
\end{align}
where $J_1(z^3)$ is an OCS on the space $\rr^4$ for which $z^3$ is
constant, and $J_0$ is the standard OCS on $\rr^2$ with coordinate
$z^3=x^3 + i y^3$. Any OCS on $\rr^4$ is necessarily constant.
Moreover, the OCSes on $\rr^4$ consistent with a fixed orientation are
parametrized by $\SO(4)/\U(2)$, the complex projective line $\cp^1$,
so $J_1$ can be regarded as a map $f:\cc\rightarrow\cp^1$. If $f$ is
holomorphic, then (\ref{wpoi}) is integrable and so defines a global
OCS on $\rr^6$. It is an example of a \emph{\wp\ orthogonal complex
  structure}, as defined in Section~\ref{warped}. While a \wp\ OCS on
$\rr^6$ is determined by a single meromorphic function, these
structures become much more complicated in higher dimensions; this
will be discussed further in Section 4.

The differential of any conformal map $\psi:\rr^6\setminus\{p\}
\rightarrow\rr^6$ lies in ${\rm{CO}}(6) = \rr_+\!\times{\rm{O}}(6)$,
thus the conformal group ${\rm{O}}_{+}(1,7)$ (time-oriented Lorentz
transformations) acts on the space of OCSes on subdomains of $\rr^6$
by conjugation by the differential $\psi_*$.  The round metric on $S^6
\setminus \{\infty\} = \rr^6$ is conformally equivalent to the
Euclidean metric $\geu$. Thus if $J$ is an OCS defined on $\rr^6$ away
from finitely many points, we can equivalently view $J$ as an OCS on
$S^6$ away from finitely many points.
\begin{definition}{\em
Let $J$ be an orthogonal complex structure  
defined on $\Omega = S^6 \setminus K$ where $K$ is a finite set of points. 
Then $J$ is said to have {\em{finite energy}} if 
\begin{align}
\label{finiteenergy}
\int_{S^6 \setminus K}\|\nabla J\|^6 dV < \infty,
\end{align}
where the covariant derivative, norm and volume form are taken with
respect to the round metric on $S^6$.}
\end{definition}
The main result in this paper shows that the above 
\wp\ construction gives all of the finite energy OCSes on $\rr^6$,
up to conformal equivalence:
\begin{theorem}
\label{t2}
Let $J$ be an orthogonal complex structure of class $C^1$ on $S^{6}
\setminus K$ with finite energy, where $K$ is a finite set of points.
\begin{enumerate}
\item[(i)]
$J$ is conformally equivalent to a \wps globally defined on
$\rr^6 = S^6 \setminus \{\infty\}$ with the correct orientation, and
determined by a rational function $f:\cc \rightarrow
\cp^1$. 
\item[(ii)]
$J$ is conformally equivalent to the standard
orthogonal complex structure on $\rr^6 = \cc^3$ if and only if $f$ is
constant. 
\item[(iii)] 
$(\rr^6, J)$ is biholomorphic to
$\cc^3$, and $(\rr^6,\geu,J)$ is cosymplectic (the K\"ahler form 
is co-closed), but is not locally conformal 
to a K\"ahler metric, in particular, it is not K\"ahler 
unless $f$ is constant.
\end{enumerate}
\end{theorem}

Taking $f:\cc\rightarrow\cp^1$ to be a non-algebraic meromorphic
function in the \wp\ construction, we find examples of infinite
energy OCSes globally defined on $\rr^6$. Taking products of these
with the standard OCS on $\rr^{2k}$, one obtains examples in all
higher dimensions which violate any reasonable finite energy
assumption.

If in particular we choose a doubly-periodic meromorphic function on
$\cc$, we find the following examples of non-standard complex flat
tori.
\begin{theorem}
  \label{torusex} Let $(T^2, J_2)$ be an elliptic curve with a
  compatible flat metric $g_2$, and let $(T^{4}, g_4)$ be a flat $4$-torus.  
  There is
  an infinite-dimensional space of complex structures on the $6$-torus
  $(T^6, g_6) = (T^4 \times T^2, g_4 \oplus g_2)$ which are orthogonal
  relative to the product metric. These structures are of \wp\ form,
  determined by a meromorphic function $f : (T^2, J_2) \rightarrow
  \cp^1$. Lifting to $\rr^{6}$, they have infinite energy.
\end{theorem}
We note that, just as in Theorem~\ref{t2} (iii), these tori are cosymplectic 
but not K\"ahler for non-constant $f$, see Proposition \ref{cosymp} below. 
A similar construction yields non-standard examples on tori in all higher even 
dimensions, see Subsection \ref{warptori}.  
However, such examples do not exist in real dimension four. Locally
conformally flat compact Hermitian surfaces have been classified by
Pontecorvo \cite{Pontecorvo}.  In the flat case, the OCS must lift to
a constant OCS on $\rr^4$, thus the space of OCSes on a flat $4$-torus
is finite dimensional. 

A complete classification of finite energy OCSes as in Theorem
\ref{t2} in higher dimensions is much more difficult. We do not
attempt such a general classification in this paper, but our work does
lead to some natural questions that may point the way towards a
possible solution of the problem. In any case, we prove a Liouville
theorem under a stronger assumption:
\begin{definition}{\em
 Let $J$ be an orthogonal complex structure 
defined on $\rr^{2n}$. We say that $J$ is \emph{asymptotically constant} if
\begin{align*}
  \|J(\mathbf{x}) - J'\|\to 0 \mbox{ as }\mathbf{x}\to\infty,
\end{align*}
for some constant orthogonal complex structure $J'$. Here, the norm
refers equally to the Euclidean or round metric, as any conformal
factor cancels out.}

\end{definition}
Assuming this condition, we have the following Liouville theorem in
all even dimensions.
\begin{theorem}
\label{t3}
Let $J$ be an orthogonal complex structure of class $C^1$ on
$\mathbb{R}^{2n}$ which is asymptotically constant. Then $\pm J$ is
isometrically equivalent to the standard constant orthogonal complex
structure on $\rr^{2n}$.
\end{theorem}
Theorem \ref{t3} will be proved in Section \ref{finding}. We next give
a brief outline of the proof of Theorem \ref{t2}. We use the twistor
fibration $\cp^3\rightarrow\Q^6\rightarrow S^6$, which was studied in
particular detail by Slupinski \cite{Slupinski}. The complex
$6$-quadric $\Q^6$ fibers over $S^6$, with fiber the complex
projective space $\cp^3$ that can be identified with $SO(6)/U(3)$, and
local sections are orthogonal almost complex structures compatible
with a fixed orientation. Such a section is integrable precisely when
the graph is a holomorphic subvariety. For simplicity, assume that $J$
is an OCS on $\rr^6 = S^6 \setminus \{ \infty \}$, so that its graph
$J(\rr^6)$ lies in $\Q^6\setminus \cp^3_{\infty}$, where
$\cp^3_\infty=\pi^{-1}(\infty)$ is the fiber over the point at
infinity.

Consider the closure $\X = \overline{J(\rr^6)} \subset \Q^6$.  The
finite energy assumption (\ref{finiteenergy}) implies that the graph
of $J$ has finite area. This in turn implies that its closure is a
$3$-dimensional analytic subvariety, by a theorem of Bishop
\cite{Bishop}. Moreover, by Chow's Theorem, it is algebraic
\cite{Chow}. Now, any $3$-dimensional subvariety $\X \subset \Q^6$ has
a bidegree $(q,p)$; see Section \ref{linears}.  Since our $\X$ arises
from an OCS, it hits generic twistor fibers in one point, and this
implies that the bidegree of $\X$ in $\Q^6$ is $(1,p)$, and the degree
of $\X$ in $\cp^7$ is $p + 1$.  Theorem \ref{t2} will then be seen as
a consequence of the following result.
\begin{theorem}
\label{texc}
Let $\X$ be a threefold of type $(1,p)$ in $\Q^6$.
Then $\X$ yields an orthogonal complex structure maximally 
defined on $S^6 \setminus E$, where $E$ is a closed set with real dimension 
at least $2$ unless $\X$ corresponds
to a \wps globally defined on $\rr^6$.
\end{theorem}
This will be proved in Section~\ref{finalsec}, using the
classification of threefolds of order one in the $6$-quadric obtained
in \cite{BV}, some of the results of which are tailored to
applications in the present paper.

In closing the Introduction, we remark that the above theorem can be
applied to give a partial classification of locally conformally flat
Hermitian threefolds. There are also applications of our theorems
to the theory of harmonic maps from Euclidean spaces. These aspects
will be discussed in a forthcoming work.

\subsection{Acknowledgements}
The authors would like to thank Vestislav Apostolov, Olivier Biquard,
David Calderbank, Antonio Di Scala, Paul Gauduchon, Brendan Hassett,
Nigel Hitchin, Nobuhiro Honda, and Max Pontecorvo for useful conversations regarding
algebraic geometry, orthogonal complex structures, and twistor
theory. We also acknowledge Shulim Kaliman for useful conversations
regarding contractible complex algebraic varieties.
\section{Background}
\label{backsec}
\subsection{Complex structures, isotropic Grassmannians and spinors}
\label{ocssec}
 There is a bijective correspondence between the following objects:\hs
(i) points of the coset space $\Z n=\SO(2n)/\U(n)$,\hs (ii) constant
or linear OCSes on $\rr^{2n}$ consistent with a fixed orientation,\hs
(iii) skew-symmetric orthogonal matrices with Pfaffian equal to
$1$,\hs (iv) maximal isotropic subspaces in $\cc^{2n}$ inducing a
fixed orientation.\\[3pt] To formalize the correspondence between (i)
and (ii), first note that the isotropy subgroup $\U(n)$ of $\Z n$ may
be regarded as the stabilizer of a fixed OCS $\bJ$ on $\rr^{2n}$.

We may express $\bJ$ as a skew-symmetric orthogonal matrix of size
$2n$.  By reducing this matrix to standard block-diagonal form (as in
(\ref{bdf}) below), we see that $\det\bJ=1$ irrespective of the
induced orientation on $\rr^{2n}$. The latter is instead encoded in
the \emph{Pfaffian} of $\bJ$. Recall that the determinant of any
skew-symmetric matrix $M$ of size $2n$ can be written
\begin{align}
\label{Pf}
\det M = (\Pf M)^2,
\end{align}
where $\Pf M$ is a polynomial of degree $n$ in its entries. Standard
expressions for the Pfaffian show that $\Pf\bJ=1$ if and only if $\bJ$
induces a positive orientation on $\rr^{2n}$. 

Any other OCS will now equal $J=A\bJ A^{-1}$ for some $A\in O(2n)$,
and $\Pf J=(\det A)(\Pf\bJ)$, so to preserve orientation we must take
$\det A=1$.  We can then map $J$ to the coset $A\U(n)\in\Z n$.  For
the description in (iv), we associate to $J$ its $+i$-eigenspace
$T^{1,0}$, a totally isotropic subspace of $\cc^{2n}$.  We shall take
the remaining mechanics of this correspondence for granted, and
explain instead how spinors can be used to add a fifth class of
objects to the list above.

Consider the complex representation 
\[\Delta=\Delta_+\oplus\Delta_-\] 
of $\Spin(2n)$, where each irreducible summand $\Delta_\pm$ has
dimension $2^{n-1}$. Given a non-zero spinor $\phi$ in $\Delta_+$,
\begin{align}
\label{Vpsi}
V_\phi = \{v \in \cc^{2n} : v \cdot \phi = 0 \}
\end{align}
is an isotropic subspace, where $\cdot$ denotes Clifford
multiplication. This follows because because if $v,w\in V_\phi$ then
\begin{align}
\label{isotropic}
0=v\cdot(w\cdot\phi)-w\cdot(v\cdot\phi)=-2\left<v,w\right>\phi.
\end{align}
Using the underlying scalar product, we can identify
$\cc^{2n}$ with its dual and regard Clifford multiplication as
an injection
\begin{align*}
m:\Delta_+\to\cc^{2n}\otimes\Delta_-.
\end{align*}
Choose a basis $(\delta_\ell)$ of $\Delta_-$, and set
\begin{align}
\label{mphi}
m(\phi)=\sum_{\ell=1}^{2^{n-1}}\alpha_\ell\otimes\delta_\ell.
\end{align} 
Then the $\alpha_\ell$ span the annihilator $(V_\phi)^\circ$ of
$V_\phi$; the bigger the latter, the smaller its annihilator. In the
extreme case, $V_\phi$ is \emph{maximal} isotropic if and only if
$(V_\phi)^\circ$ is \emph{isotropic} (of the same dimension). A nice
way to chararacterize when this occurs is through representation
theory.

It is well-known that an element of the Clifford algebra
$\mathrm{Cl}(\cc^{2n})$ is itself an endomorphism of $\Delta$. Furthermore,
Cartan established equivariant decompositions
\begin{align}
\label{dec3}
\begin{split}
  \Delta_+ \otimes \Delta_+ &= \Lambda^n_+\>\oplus\,\Lambda^{n-2} \oplus
  \cdots\\[3pt]
 \Delta_+ \otimes \Delta_- &= \Lambda^{n-1} \oplus \Lambda^{n-3} \oplus \cdots,
\end{split}
\end{align}
where $\Lambda^k= \Ext^k(\cc^{2n})$ denotes exterior power of the
basic representation $\Lambda^1=\cc^{2n}$ of
$\SO(2n)=\Spin(2n)/\mathbb{Z}_2$, and $\Lambda^n_+$ is the $+1$
eigenspace of the Hodge map $*\colon\Lambda^n\to\Lambda^n$.

\begin{theorem}[Cartan \cite{Cartan}]
\label{Cartant1}
The isotropic subspace $V_\phi$ is maximal if and only if the only
non-zero component of $\phi\otimes\phi$ is in $\Lambda^n_+$.  This
component is decomposable (that is, a simple $n$-form), and generates
the subspace corresponding to $\phi$.
\end{theorem}

If the condition of the theorem is satisfied, then $\phi$ is called a
\emph{pure} spinor.  Sometimes we shall use the symbol $\phi$ to
indicate the projective class of a pure spinor, and we let $J_\phi$
denote the OCS (or skew-symmetric orthogonal matrix with positive
Pfaffian) characterized by
\begin{align*}
T^{1,0}=V_\phi,\quad \Lambda^{0,1}=(V_\phi)^\circ.
\end{align*}
There are no purity conditions for $n =2,\,3$ because in both cases
$\Lambda^n_+$ equals the symmetric part of the tensor product; hence
the coset spaces in (ii) are complex projective spaces:
\begin{align*}
\Z2= \cp^1,\qquad\Z3= \cp^3.
\end{align*}
However, in dimension $8$, there is one
quadratic relation given by projection to the summand $\Lambda^0$, and
\begin{align}
\label{so8}
\Z4 \subset \cp(\Delta_+)
\end{align}
is a non-degenerate quadric $\Q^6\subset \cp^7$.

\subsection{Twistor fibrations}
\label{tfibsec}
We shall first discuss the twistor space $\sZ=\sZ(\rr^{2n})$ of
Euclidean space $\rr^{2n}$. As a smooth manifold, it is the product
\begin{align}
\label{sZ}
\sZ =\Z n \times\rr^{2n}.
\end{align}
The ``twistor'' complex structure $\sJ$ on $\sZ$ is defined as
follows. The tangent space to $\sZ$ at a point $p=(J,\mathbf{x})$
splits as
\begin{align}
\label{VH}
T_p\sZ=V_p \oplus H_p,
\end{align}
where the vertical space $V_p$ is tangent to the $\Z n$ factor at $J$,
and the horizontal space $H_p$ is tangent to the $\rr^{2n}$ factor at
$\mathbf x$. As the notation suggests, $J\in\Z n$ is itself an almost
complex structure on the vector space $\rr^{2n}\cong H_p$. 

The tangent space to $\Z n$ at $J$ can be identified with those
skew-symmetric endomorphisms of $\rr^{2n}$ that \emph{anti-commute}
with $J$. It follows that, if $H_p^{1,0}$ denotes the $+i$-eigenspace
for $J$, there is a canonical identification
\begin{align}
\label{L20}
V_p\otimes_\rr\cc=\Ext^2(H_p^{1,0})\oplus\Ext^2(H_p^{0,1}).
\end{align}
This not only determines the standard complex structure of $\Z n$ but
also allows us to fix its sign in the context of the twistor
fibration $\sZ\to\rr^{2n}$. We define $V_p^{1,0}=\Ext^2(H_p^{1,0})$, and decree
the $+i$-eigenspace of $\sJ$ to be
\begin{align}
\label{decree}
V_p^{1,0}\oplus H_p^{1,0}\subset (T_p\sZ)\otimes_\rr\cc.
\end{align}
It is known that, with this careful choice, $\sJ$ is integrable
\cite{AHS,Besse,deBartolomeisNannicini,ObrianRawnsley,Salamon1985}.

An analogous construction can be used to define the twistor space of
any Riemannian manifold. The fiber over each point is again $\Z n$.
The splitting (\ref{VH}) is accomplished by means of the Levi-Civita
connection, and this enables one to define a tautological almost
complex structure $\sJ$ in the same way. In
particular, the twistor space of the even-dimensional round sphere is the
total space $\sZ(S^{2n})$ of the fibration
\begin{align}
\label{twsph} 
\Z n\rightarrow\sZ(S^{2n})\rightarrow S^{2n},
\end{align}
endowed with the structure $\sJ$. The orthonormal frame bundle of
$S^{2n}$ is the Lie group $\SO(2n + 1)$, so (\ref{twsph}) is the
fibration
\begin{align*}
\SO(2n) / \U(n)  \rightarrow \SO(2n + 1) / \U(n) \rightarrow S^{2n}.
\end{align*}
This was used in the study of minimal surfaces in $S^{2n}$
\cite{Calabi, Barbosa}. On the other hand, it is known that $\SO(2n +
1) / \U(n) \cong \SO( 2n + 2) / \U(n+1)$ 
(see \cite{BattagliaS6, SalamonOCS}), and
so there is a fibration
\begin{align}
\label{twist1}
\Z n \rightarrow\Z{n+1}\stackrel\pi\rightarrow S^{2n}.
\end{align} 
We shall give another description of this in Section~\ref{integ}.

Over the $4$-sphere, one recovers the ``Penrose fibration'' 
\begin{align} 
\label{Penrose}
\cp^1\rightarrow \cp^3 \rightarrow S^4.
\end{align}  
If we identify $S^4$ with the quaternionic projective line
$\mathbb{HP}^1$, this is merely a Hopf-type map. In dimension $6$, we
have
\begin{align} 
\label{tproj6}
\cp^3 \rightarrow \Q^6\rightarrow S^6,
\end{align}
as stated in the Introduction, although we now have the more precise
description $\Q^6\subset\cp(\Delta_+)$. Pending an explicit formula
for the projection to $\rr^6\subset S^6$ in Section
\ref{twistorprojection}, we shall study the geometry underlying
(\ref{tproj6}) in the next subsection.

\subsection{Linear spaces on $6$-quadrics}
\label{linears}
In this subsection, we abbreviate to $\Lambda$ the standard complex
representation $\Lambda^1=\cc^8$ of
$\SO(8)=\Spin(8)/\mathbb{Z}_2$. Just as $\Q^6\subset\cp(\Delta_+)$
parametrizes maximal positively-oriented isotropic subspaces of
$\Lambda$, so the $6$-quadric in $\cp(\Delta_-)$ parametrizes maximal
negatively-oriented isotropic suspaces in $\Lambda$. The triality
principal asserts that the representations
$\Delta_+,\,\Delta_-,\,\Lambda$ are equivalent by a cyclic permutation
induced by an outer automorphism of $\Spin(8)$, and we deduce that the
$6$-quadrics in $\PP(\Delta_-),\,\PP(\Lambda)$ parametrize different
families of maximal isotropic subspaces of $\Delta_+$ or,
equivalently, linear $\cp^3$-s in the twistor space (\ref{so8}).

In this way, we see the classical fact that the set of $\cp^3$-s in
the twistor space has two components, each of which can itself be
identified with a $6$-quadric. In the twistor context, this theory was
described by Slupinski \cite{Slupinski}.  The family parametrized by
the quadric in $\cp(\Lambda)$ contains the twistor fibers (themselves
parametrized by the real submanifold $S^6\subset\cp(\Lambda)$) but
consists of what generally we shall call ``vertical'' $\cp^3$-s. A
vertical $\cp^3$ is either a fiber or a twistor space of a $4$-sphere
conformally embedded in $S^6$ (via (\ref{Penrose}) or a
negatively-oriented version) \cite{Slupinski}.  On the other hand, the
family parametrized by the quadric in $\cp(\Delta_-)$ consists of
``horizontal'' $\cp^3$-s. If $P$ is a horizontal $\cp^3$ then there
exists a unique $p\in S^6$ such that $P\cap\pi^{-1}(p)$ is a $\cp^2$
and $P\cap\pi^{-1}(q)$ a single point for $q\ne p$. Moreover, $P$ is
determined by $p$ and $P\cap\pi^{-1}(p)$, so the family of horizontal
$\cp^3$-s is a ``dual'' twistor space projecting to $S^6$ with fiber
$(\cp^3)^*$ (we have after all just swapped $\Delta_+$ and
$\Delta_-$).

We continue to denote the twistor space of $S^6$ by $\Q^6$, leaving
implicit its embedding in $\cp(\Delta_+)$. The homology group
\begin{align}
\label{homology}
H_6(\Q^6,\mathbb{Z}) = \mathbb{Z} \oplus \mathbb{Z}
\end{align}
is generated by a horizontal $\cp^3$ (we choose this to represent the
first factor), and a vertical $\cp^3$ (the second factor). This
implies that any $3$-dimensional subvariety $\X \subset \Q^6$ has a
well-defined \emph{bidegree} $(q,p)$. Any vertical or horizontal
$\cp^3$ has zero self-intersection in $\Q^6$, and $\cp^3$-s from
opposite families have intersection $1$.

To illustrate this from Cartan's viewpoint, we use the isomorphisms
\begin{align}
\label{dec33}
\Delta_- \otimes \Delta_- 
&=\ \Lambda^4_-\oplus\Lambda^2\oplus\Lambda^0,\\
\Delta_-\otimes \Delta_+
&=\ \Lambda^3\oplus\Lambda^1.
\label{dec44}
\end{align}
They are instances of (\ref{dec3}) though here (having applied
triality) $\Lambda^k=\Ext^k\Delta_+$.  Two horizontal $\cp^3$-s are
determined by pure spinors $\phi,\,\psi\in \Delta_-$. The non-zero
component of $\phi \otimes \psi$ in the smallest summand of
(\ref{dec33}) will always be a simple form that spans the intersection
of the corresponding $4$-dimensional subspaces in $\Delta_+$. In the
generic case, the component $\langle\phi,\psi\rangle\in\Lambda^0$ will
be non-zero, and the two $\cp^3$-s will have empty intersection.  If
$\langle\phi,\psi\rangle=0$ then the component in $\Lambda^2$
(formally $\phi\wedge\psi$) will be a simple wedge product of 1-forms,
indicating that the two $\cp^3$-s intersect in a $\cp^1$.  Similarly,
two $\cp^3$-s from different families will intersect generically in a
point or (if the $\Lambda^1$ component of $\phi\otimes\psi$ in
(\ref{dec44}) vanishes) a $\cp^2$. A similar intersection criterion
holds in higher dimensions, but we will not require it. We refer the
reader to \cite{Cartan,Chevalley} for further details.\smallbreak

The following lemma will be used in Section~\ref{finalsec}. It
describes the geometry of the twistor projection restricted to a plane
$\cp^2$ in $\Q^6$.
\begin{lemma} 
\label{cp2}
Every linear $\cp^2 \subset \Q^6$ is either contained entirely in a fiber of the
twistor projection, or intersects exactly one twistor fiber in a
$\cp^1$ and all other fibers in a point or the empty set.
\end{lemma}
\begin{proof}
Given any $\cp^2$, call it $P$, there is exactly one horizontal and
one vertical $\cp^3$ containing $P$ \cite[Proposition 3.2]{BV}. 
Consider the horizontal $\cp^3$ containing $P$.  As mentioned above, 
a horizontal $\cp^3$ in $\Q^6$ hits exactly one fiber in a $\cp^2 = P_0$, 
and hits every other fiber in a point. The planes $P$ and $P_0$ are then two
$\cp^2$-s in a $\cp^3$; they are either equal or intersect in a
$\cp^1$.
\end{proof}

\begin{remark}{\em
  If we look instead at the vertical $\cp^3$ containing $P$, it is
  either a twistor fiber or (from above) can identified with the
  twistor space of an $S^4 \subset S^6$. Any $\cp^2$ in this twistor
  bundle hits exactly one fiber in a $\cp^1$ and hits every other
  fiber over this $S^4$ in exactly one point \cite[Proposition
    3.3]{SalamonViaclovsky}.}
\end{remark}

\section{Coordinates on the twistor fiber}
\label{holocor}
We shall return to consider the twistor space $\sZ(\mathbb{R}^{2n})$
in Section~\ref{integ}. But we first describe an atlas of coordinates
covering the space $\Z n = \SO(2n)/\U(n)$ that constitutes the twistor
fiber over $\rr^{2n}$ or $S^{2n}$. We assume that $n\ge2$.

We will define quantities
\begin{align*}
\xi_{i_1 \dots i_p}\quad &\hbox{for $p=0,2,\ldots\ $ even, } 
\hbox{up to $n$ or (if $n$ is odd) $n\!-\!1$,}\\[3pt]
\eta_{i_1\dots i_q}\quad &\hbox{for $q=1,3,\ldots\ $ odd, }
\hbox{up to $n\!-\!1$ or (respectively) $n$,}
\end{align*} 
both skew-symmetric in all indices running from $1$ to $n$. The
$\xi$-s will be holomorphic coordinates on $\Z n$. At each point of
$\Z n$ the $\eta$-s will be elements of
\begin{align}
\label{C2n}
\cc^{2n}=\rr^{2n} \otimes\cc = \Lambda^{1,0} \oplus \Lambda^{0,1},
\end{align}
decomposed relative to the standard complex structure $\bJ$ on
$\rr^{2n}$ for which $\Lambda^{1,0}$ is spanned by $dz^1,\ldots,dz^n$.

Having fixed $\bJ$, there are isomorphisms
\begin{align}
\label{Delta+ext}
\Delta_+\> &\cong\>\Lambda^{0,0}\oplus\Lambda^{2,0}\oplus\Lambda^{4,0}
\oplus\cdots\\[3pt]
\Delta_-\> &\cong\>\Lambda^{1,0}\oplus\Lambda^{3,0}\oplus\Lambda^{5,0}
\oplus\cdots
\label{Delta-ext}
\end{align}
(Strictly speaking, we also need a trivialization of $\Lambda^{n,0}$
that removes the distinction between $\Lambda^{p,0}$ and
$\Lambda^{0,n-p}$.) Rather than adopt an overtly invariant approach,
we shall merely use these decompositions to motivate Cartan's
technique.

First, the $\xi_{i_1 \dots i_p}$ represent the components of a spinor
$\bxi\in \Delta_+$ relative to a basis compatible with
(\ref{Delta+ext}).  We arrange them into groups
\begin{align}
\label{groups}
\xi_0,\quad \xi_{ i_1 i_2}\ (i_1 < i_2),\quad
\xi_{i_1 i_2 i_3 i_4}\ (i_1 < i_2 < i_3 < i_4),\quad\cdots
\end{align}
and order them lexicographically within each group, to give a total of
\begin{align*}
N=\sum_{p\ \mathrm{even}} \binom np = 2^{n-1}
\end{align*}
scalars. For example, $\xi_0$ (here $p=0$ so logically the subscript
is $\emptyset$) represents the component of $\bxi$ in the trivial
summand $\Lambda^{0,0}$.

Next, we use (\ref{mphi}) and a compatible basis $(\delta_\ell)$ of
(\ref{Delta-ext}) to convert $\bxi=\phi$ into 1-forms $\alpha_\ell$
for $\ell=1,\ldots,N$. The $\eta$-s are precisely these $1$-forms, but
rearranged to respect (\ref{Delta-ext}). The summand $\Lambda^{1,0}$
gives us the first $n$ of them, namely
\begin{align}
\label{etai}
\eta_i &= \xi_0 dz^i - 
\suml_{k=1}^n\xi_{ik} d\bz^k,\qquad i=1,\ldots,n.
\end{align} 
This formula reflects the fact that Clifford multiplication by a
vector $\pd/\pd z^i$ or $\pd/\pd\bz^j$ acts on (\ref{Delta+ext}) as
the sum of an exterior and interior product respectively. (To make
sense of this, it is easiest to regard the summands of $\Delta_\pm$ as
exterior powers of \emph{vectors}.) The skew-symmetry guarantees that
the $\eta_i$ span an \emph{isotropic} subspace in (\ref{C2n}). If
$\xi_0\ne0$, this subspace will be maximal and thereby define $J\in\Z n$.

In general, the matrix $(\xi_{ij})$ defines an element of the
\emph{tangent space} to $\Z n$ at $J$, a $\U(n)$-module identified
with $\Lambda^{2,0}$; cf.~(\ref{L20}). The point of $\Z n$ determined
by (\ref{etai}) with $\xi_0=1$ is parametrized by an affine space, and
it is evident that it cannot cover all of $\Z n$.  If $\xi_0 = 0$ and
$n\ge3$, the $\eta_i$ are not even linearly independent, and we need
to add more forms.  We are therefore forced to consider $1$-forms
arising from further summands in (\ref{Delta-ext}). Applying interior
and exterior products, we obtain the additional elements
\begin{align}
\label{etaiq}
\eta_{i_1 \dots i_q} = 
\suml_{k=1}^q (-1)^{k-1} \xi_{i_1 \dots \hat{i_k} \dots i_q}
dz^{i_k} - \suml_{m=1}^{n} \xi_{i_1 \dots i_q m} d\bz^m,
\ \quad q\ge3\hbox{ odd},
\end{align}
where the notation $\hat{i_k}$ means to omit this index.

In order that $\bxi$ be a pure spinor, the $\eta$-s (recall these are
the $\alpha_\ell$-s in (\ref{mphi})) must span an \emph{isotropic}
subspace.  In particular, the forms (\ref{etai}) and (\ref{etaiq})
must be mutually isotropic. This provides us with the scheme of
equations
\begin{align}
\label{quadrics}
\xi_{0}\xi_{i_1 \dots i_p}\ =\ \suml_{k=1}^{p-1}
(-1)^{k-1}\xi_{i_k i_p} \xi_{i_1 \dots \hat{i_k} \dots i_{p-1}},
\qquad p\ge4\hbox{ even}.
\end{align}
(The last equation is a tautology if $p=2$, which helps to check the
signs.) If $\xi_0=1$, we have
\begin{align*}
\eta_{i_1 \dots i_q} = \suml_{k=1}^q (-1)^{k-1} \xi_{i_1 \dots
\hat{i_k} \dots i_q}\eta_{i_k},
\end{align*}
and in this case isotropy is manifest.

\begin{remark}\label{exp}
{\em
If $\xi_0\ne0$, then the remaining components of $\bxi$ in
(\ref{Delta-ext}) are determined by its projection $\beta$ to
$\Lambda^{2,0}$. Indeed, it follows from (\ref{quadrics}) that $\bxi$
can be identified with
\begin{align*}
\textstyle
\xi_0\exp\beta=\xi_0\big(1+\beta+\frac12\beta\wedge\beta+\cdots
\big).
\end{align*}
(This fact is well-known in the context of generalized complex
structures; see for example \cite{Gualtieri}.) More generally, $\bxi$
will have the form $\gamma\wedge\exp\beta$ for some
$\gamma\in\Lambda^{2k}$.}
\end{remark}

Let $\Y_n\subset\cp^{N-1}$ be the intersection of quadrics defined by
the equations (\ref{quadrics}) with $p\ge4$. One can show that these
\begin{align*} 
  \tilde N\>=\sum_{p\>\mathrm{even} > 2} \binom np= N - \binom n2 -1
\end{align*}
equations are independent. We now define
\begin{align}
\label{FF}
\FF:\Y_n \cap \{ \xi_0 \ne0 \} \rightarrow\Z n,
\end{align} 
by mapping $[\bxi]$ to the maximal isotropic subspace
$(V_\sbxi)^\circ$ of $\cc^{2n}$.  More explicitly,
\begin{align*}
\begin{array}{rcll}
[\xi_0,\xi_{12},\ldots,\xi_{1\cdots n}] &\mapsto&
\Span\{\eta_1,\ldots,\eta_{123},\ldots,\eta_{2\cdots n}\}
\quad & \hbox{if $n$ is even},\\[4pt]
[\xi_0,\xi_{12},\ldots,\xi_{2\cdots n}] &\mapsto&
\Span\{\eta_1,\ldots,\eta_{123},\ldots,\eta_{1\cdots n}\}
& \hbox{if $n$ is odd}.
\end{array}
\end{align*}
There are $\tilde N$ equations in $\cp^{N-1}$ defining $\Y_n$, so $\Y_n$ is
a real $n(n-1)$-dimensional manifold. On the other hand, the real
dimension of $\Z n$ equals
\begin{align*}
  \dim \SO(2n) - \dim \U(n) = n (2n-1) - n^2 = n(n-1).
\end{align*}
This implies that the image of (\ref{FF}) is an open subset of
$\Z n$.  

We can remove the restriction $\xi_0\ne0$ as follows. In the general
case, $\Y_n$ will have several irreducible components. We need to add
all of the quadratic relations from Theorem~\ref{Cartant1}; these
additional relations will specify a unique irreducible component of
$\Y_n$ that we will call $\Y_n^+$. This is analogous to the
well-known Pl\"ucker relations for the orthogonal Grassmannians. In
general, the map
\begin{align}
\label{psi+}
\FF: \Y_n^+\rightarrow\Z n,\quad [\bxi]\mapsto(V_\sbxi)^\circ
\end{align} 
is well-defined, which allows the possibility that $\xi_0 = 0$. 
\begin{theorem}
  The map {\rm(\ref{psi+})} is a biholomorphism, where $\Z n$ has the
  complex structure as a Hermitian symmetric space, and $\Y_n^+$ has
  the induced complex structure as a complex submanifold of
  $\cp^{N-1}$.
\end{theorem}
\begin{proof}
This is well-known, see \cite{Inoue}.
\end{proof}

We will henceforth use the map $f$ to identify $\Y_n^+$ and $\Z n$. 

\subsection{Low-dimensional cases.}
\label{lowdims}
This subsection summarizes the situation for $n$ equal in succession to
$2,\,3,\,4$, as this helps to clarify the above discussion. In each
case, we identify points of $\Z n$ with positively-oriented maximal
isotropic subspaces of $\cc^{2n}$.\smallbreak

For $n=2$, we have 
\begin{align}
\label{etaeta} 
\eta_1= \xi_0 dz^1 - \xi_{12} d\bar{z}^2,\qquad
\eta_2= \xi_0 dz^2 + \xi_{12} d\bar{z}^1.
\end{align}
The biholomorphism $\FF:\Y_2^+\to\Z2\cong\cp^1$ is given by
\begin{align*}
[\xi_0, \xi_{12}] \mapsto \Span\{ \eta_1, \eta_2 \},
\end{align*}
and there is no relation between $\xi_0,\xi_{12}$.\smallbreak

Now suppose that $n=3$. In addition to 
\begin{align}
\begin{split}
\label{etai6}
\eta_1 &= \xi_0dz^1-\xi_{12}d\bz^2-\xi_{13}d\bz^3,\\
\eta_2 &= \xi_0dz^2-\xi_{23}d\bz^3+\xi_{12}d\bz^1,\\
\eta_3 &= \xi_0dz^3+\xi_{13}d\bz^1+\xi_{23}d\bz^2,
\end{split}
\end{align}
defined by (\ref{etai}), we have 
\begin{align}
\label{etaiii}
\eta_{123} =  \xi_{23} dz^1 - \xi_{13} dz^2 + \xi_{12} dz^3,
\end{align}
since $\xi_{ijkl}=0$ here. The biholomorphism $\FF:\Y_3^+\rightarrow
\Z3\cong\cp^3$ is given by
\begin{align*}
[\xi_0, \xi_{12}, \xi_{13}, \xi_{23}] \mapsto \Span
\{ \eta_1, \eta_2, \eta_3, \eta_{123} \},
\end{align*}
and again there is no constraint.\smallbreak

Finally, suppose that $n=4$, so that (\ref{etaiii}) is upgraded to
\begin{align*} 
\eta_{ijk} = \xi_{jk} dz^{i} - \xi_{ik} dz^j + \xi_{ij} dz^k
+ \suml_m\xi_{ijkm} d\bz^m. 
\end{align*}
The biholomorphism $\FF:\Y_4^+\rightarrow\Z4$ is given by 
\begin{align*}
[\xi_0, \xi_{12}, \xi_{13}, \xi_{14}, \xi_{23}, \xi_{24}, \xi_{34}, \xi_{1234}]
\mapsto \Span \{ \eta_1, \eta_2, \eta_3,
\eta_4, \eta_{123}, \eta_{124}, \eta_{134}, \eta_{234} \},
\end{align*}
and there is now a single quadratic relation, 
\begin{align}
\label{1324} 
  \xi_0\xi_{1234} = \xi_{12}\xi_{34}-\xi_{13}\xi_{24}+\xi_{14}\xi_{23}
\end{align}
confirming that $\Z4$ is a nonsingular quadric hypersurface in
$\cp^7$. However, if $\xi_0=0$ then the $\xi_{ijkm}$ are independent of
the $\xi_{ij}$.

\subsection{Skew-symmetric orthogonal matrices}
\label{skewsym}
In this subsection, we make explicit the map taking a pure spinor
$\phi$ to a skew-symmetric matrix $J_\phi$ with Pfaffian $1$,
discussed at the start of Subsection~\ref{ocssec}.

For the most part, and for the sake of simplicity, we explain the
construction in dimension 6. For any
$[\xi_0,\xi_{12},\xi_{13},\xi_{23}]\in\cp^3=\Z3$, the associated
maximal isotropic space of $(1,0)$-forms is spanned by (\ref{etai6})
and (\ref{etaiii}). Its annihilator $T^{0,1}$ is spanned by
\begin{align}
\begin{split}
\label{vvv}
v_1 &= \xi_0\od_1 - \xi_{12}\pd_2 - \xi_{13}\pd_3\\ 
v_2 &= \xi_0\od_2 - \xi_{23}\pd_3 + \xi_{12}\pd_1\\
v_3 &= \xi_0\od_3 + \xi_{13}\pd_1 + \xi_{23}\pd_2
\end{split}
\end{align}
where $\pd_i=\pzi$ and $\od_j=\pzbj$, together with
\begin{align*}
v_{123} & = \xi_{23}\od_1-\xi_{13}\od_2 + \xi_{12}\od_3.
\end{align*}
If $\xi_0\ne0$, then $v_1, v_2, v_3$ suffice. On the other hand, if
(for example) $\xi_{12}\ne0$, $T^{0,1}$ is spanned by
\begin{align}
\begin{split}
\label{perm}
v_2 &=  \>\xi_{12}\pd_1 + \xi_0\od_2 - \xi_{23}\pd_3\\
v_1 &= \!-\xi_{12}\pd_2 - \xi_{13}\pd_3 + \xi_0\pd_1\\
v_{123}&=\>\xi_{12}\od_3 + \xi_{23}\od_1 - \xi_{13}\od_2 
\end{split}
\end{align}
We have arranged (\ref{perm}) so that the basis vectors coincide with
those in (\ref{vvv}) after swapping $\od_i\leftrightarrow\pd_i$ for
$i=1,2$. An inspection of the new coefficients yields
\begin{proposition}
  Let $J=J_\phi$ be the skew-symmetric orthogonal matrix obtained from
  a projective spinor $\phi=[\xi_0,\xi_{12},\xi_{13},\xi_{23}]$, and
  $J'=J_{\phi'}$ the matrix obtained from
  $\phi'=[\xi_{12},-\xi_0,\xi_{23},\xi_{13}]$. Then $J' = AJA^{-1}$,
  where $A\in \SO(2n)$ is the matrix corresponding to the
  transformation $z^1\mapsto\bz^1$ and $z^2\mapsto\bz^2$.
\end{proposition}

Given this proposition (and analogues for other permutations of the
spinor coordinates), we can now concentrate on the case in which
$T^{0,1}$ is spanned by (\ref{vvv}) and we can take $\xi_0=1$. We may
write
\begin{align*}
v_i = w_i+\sqrt{-1}Jw_i,
\end{align*}
where $w_i$ are \emph{real} vectors. In this basis, the OCS is
represented by the block diagonal matrix
\begin{align}
\label{bdf}
\bJ = \mbox{diag}(J_0, \dots, J_0),
\end{align}
where
\begin{align}
  \label{J0}
J_0 =
\left(
\begin{matrix}
0 & -1 \\
1 & 0  \\
\end{matrix}
\right)
\end{align}
represents the standard complex structure on $\rr^2$. 

We let $A$ denote the real matrix corresponding to 
the basis change
\begin{align*}
  ( w_1,Jw_1, w_2,Jw_2, w_3,Jw_3)^\top= A \Big(
  \frc{\partial}{\partial x_1}, \frc{\partial}{\partial y_1},
  \frc{\partial}{\partial x_2}, \frc{\partial}{\partial y_2},
  \frc{\partial}{\partial x_3}, \frc{\partial}{\partial y_3}
  \Big)^{\!\top}.
\end{align*}
A computation shows that 
\begin{align*}
A = 
\left(
\begin{matrix}
-1  &  0  &   f_{12} &  g_{12}  &  f_{13}  & g_{13} \\
0   & -1  &   g_{12} & -f_{12}  &  g_{13}  & - f_{13} \\
-f_{12} & - g_{12} & -1 & 0 & f_{23}  &   g_{23}   \\
-g_{12} & f_{12}  & 0  & -1 &   g_{23}  &  -f_{23} \\
-f_{13} & - g_{13}  &  - f_{23}  & -g_{23} & -1  & 0 \\
-g_{13}  & f_{13} &  - g_{23} & f_{23}   & 0  &  -1 \\
\end{matrix}
\right), 
\end{align*}
where $\xi_{ij} = f_{ij} + \sqrt{-1} g_{ij} $. This discussion yields
the

\begin{proposition}
  The skew-symmetric orthogonal matrix corresponding to the projective
  spinor $\phi=[1, \xi_{12}, \xi_{13}, \xi_{23}]$ is
\begin{align*}
  J_\phi = A\bJ A^{-1}.
\end{align*}
\end{proposition}

We will not write out the entire formula here, but we note that the
matrix
\begin{align*}
(1 + |\xi_{12}|^2 + |\xi_{13}|^2 + | \xi_{23}|^2)J_\phi
\end{align*}
has quadratic entries in the $f_{ij}$ and $g_{ij}$,
which is straightforward to verify. 
We next consider a special case.
\begin{proposition}
  If $\phi=[1,\xi_{12}, 0, 0]$ then $J_\phi$ is a product OCS of the
  form
\begin{align*}
J = J(\xi_{12})  \oplus J_0, 
\end{align*}
where $J(\xi_{12})$ is the linear OCS on $\rr^4 = \{(z^1,
z^2,0):z^i\in\cc\}$ corresponding to $[1,\xi_{12}]$, and $J_0$ acts on
$\rr^2 = \{(0,0,z^3):z^3\in\cc \}$ as in {\rm(\ref{J0})}.
\end{proposition}
\begin{proof}
  A computation shows that as a matrix, $J_\phi$ equals
\begin{align*}
\frac1{1 + |\xi_{12}|^2} 
\left(
\begin{matrix}
0  & |\xi_{12}|^2 - 1 & -2 g_{12} & 2 f_{12} & 0 & 0  \\
1 - |\xi_{12}|^2 & 0 & 2 f_{12} & 2 g_{12} & 0 & 0 \\
2 g_{12} & - 2f_{12} & 0 & |\xi_{12}|^2 -1 & 0 & 0 \\
-2 f_{12} & -2 g_{12} & 1 - |\xi_{12}|^2 & 0 & 0 & 0 \\
0 & 0 & 0 & 0 & 0 & - (1 +|\xi_{12}|^2) \\
0 & 0 & 0 & 0 & 1 + |\xi_{12}|^2 & 0 \\
\end{matrix}
\right),
\end{align*}
which implies the proposition. 
\end{proof}

In higher dimensions we have the following. Let
\begin{align}
\label{bxi}
\bxi = (\xi_0,\xi_{12},\ldots,\xi_{1234},\ldots,\xi_{\cdots n})
\in\Delta^+
\end{align}
be a pure spinor, with skew-symmetric components (\ref{groups}).

\begin{proposition}
\label{gendim}
Suppose that all the components of {\rm(\ref{bxi})} that contain an
index $n$ vanish. Then the associated OCS has the form
\begin{align*}
J_\sbxi = J_{2n-2}  \oplus J_0, 
\end{align*}
where $J_{2n-2}$ is the linear OCS on $\rr^{2n-2} = \{(z^1, \dots
z^{n-1},0):z^i\in\cc\}$, corresponding to the spinor with components
$\xi_{i_i\cdots i_k}$ with $1 \le i_1 < \dots < i_k \le n-1$.
\end{proposition}

\begin{proof}
The proof is a computation, analogous to the previous 
proposition.
\end{proof}

\section{Integrability of the twistor space}
\label{integ}

We next explain how to construct a complete set of holomorphic
coordinates on the twistor spaces
\begin{align*}
\sZ=\sZ(\rr^{2n})\subset\sZ(S^{2n})=\Z{n+1}
\end{align*}
discussed in Subsection~\ref{tfibsec}. In a sense, we have already
done this for $\Z{n+1}$, but we must now base this construction on the
twistor fiber $\Z n$ so that it is compatible with the fibration
(\ref{twist1}).

Let $\Delta_\pm$, $\tilde\Delta_\pm$ denote the spinor bundles for
$\Spin(2n)$, $\Spin(2n+2)$ respectively. When we reduce to the former 
group, it is well-known that $\tilde\Delta_+$ restricts to the total
spin representation, so that
\begin{align}
\label{tildeD}
\Z{n+1}\subset\PP(\tilde\Delta_+)=\PP(\Delta_+\oplus\Delta_-).
\end{align}
The idea is to characterize elements of $\Z{n+1}$ in these terms, and
then translate (\ref{Delta+ext}) and (\ref{Delta-ext}) into explicit
formulae.

\begin{theorem}
\label{Clifford}
Suppose that $0\ne\phi\in\Delta_+$ and $\psi\in\Delta_-$. Then
$[\phi,\psi]\in \Z{n+1}$ if and only if $[\phi]\in\Z n$ and
$\psi=\vv\cdot\phi$ for some \emph{unique} $\vv\in\rr^{2n}$.
\end{theorem}

\noindent Recall that $\cdot$ denotes Clifford multiplication, and
that the assertion $[\phi]\in\Z n$ means that $\phi$ is a \emph{pure}
spinor.

\begin{proof} First observe that if $\vv\cdot\phi=\vv'\cdot\phi$ with 
$\vv,\vv'\in\rr^{2n}$ then $\vv-\vv'\in V_\phi$ which forces
  $\vv=\vv'$ (recall (\ref{isotropic})). So uniqueness is immediate.

We next prove that if $[\phi]\in Z_n$ then (relative to
(\ref{tildeD})) $[\phi,\,\vv\cdot\phi]\in Z_{n+1}$. Fix a pure spinor
$\phi$, and $\vv\in\rr^{2n}$. We may choose a basis $(\ee_i)$ of
$\rr^{2n}$ such that $\vv=\ee_{2n}$ and the annihilator of $\phi$ is
spanned by $\alpha_k=\ee_{2k-1}+i\kern1pt \ee_{2k}$ for
$k=1,\ldots,n$. A quick calculation reveals that the annihilator of
the Clifford product $\vv\cdot\phi$ is spanned by
$\alpha_1,\ldots,\alpha_{n-1},\overline{\alpha}_n$, and it follows
that $\vv\cdot\phi$ is a pure spinor in $\Delta_-$.

Using tildes to refer to $\rr^{2n+2}$, we want to show that the
``paired spinor''
\[\wt\phi= (\phi,\ \vv\cdot\phi) \in \wt\Delta_+\] is also pure.
To do this, we shall apply Theorem~\ref{Cartant1} to
$\wt\phi\otimes\wt\phi$, which we need to show belongs to the
underlined summand in
\begin{align}
\label{S2D}
S^2(\wt\Delta_+)\>\cong\>\underline{\wt\Lambda^{n+1}_+}\oplus
\wt\Lambda^{n-3}\oplus\wt\Lambda^{n-7}\oplus\cdots
\end{align}
(where $\wt\Lambda^\ell$ is absent if $\ell<0$). We shall do this by
considering equivariant mappings between irreducible $G$-modules,
where $G=\Spin(2n)$. Schur's lemma is the assertion that any such
mapping $f$ is either zero or an isomorphism (and is obvious since the
kernel and image of $f$ are $G$-invariant subspaces).

Now $\wt\phi\otimes\wt\phi$ is a sum of
\begin{align}
\label{LL}
\phi\otimes\phi\in\Lambda^n_+,\qquad
      (\vv\cdot\phi)\otimes(\vv\cdot\phi) \in \Lambda^n_-,
\end{align} 
and the ``mixed'' product, itself a contraction of
\begin{align}
\label{RL}
\vv\otimes(\phi\otimes\phi)\in\rr^{2n}\otimes\Lambda^n_+.
\end{align}
Using an algorithm to commute the irreducible summands of a tensor
product as in \cite{Fegan}, we see that the last space contains only
one $G$-summand isomorphic to an exterior power, namely
$\Lambda^{n+1}\cong\Lambda^{n-1}$. The exterior powers
$\Lambda^k=\Ext^k(\rr^{2n})$ for $0\le k\le n\!-\!1$, together with
$\Lambda^n_+$ and $\Lambda^n_-$, are of course distinct irreducible
$G$-modules.

Each right-hand summand of (\ref{S2D}), other than the first,
decomposes as
\[\wt\Lambda^k\>=\>\Ext^k(\rr^{2n}\oplus\rr^2)\>\cong\>
\Lambda^k\oplus (\Lambda^{k-1}\otimes\rr^2)\oplus \Lambda^{k-2},\] and
is the sum of at most four exterior powers, each of degree no greater
than $n-3$. It follows that the $G$-equivariant projections of the
elements (\ref{LL}) and (\ref{RL}) in (\ref{S2D}) do indeed all lie in
the first summand $\wt\Lambda^{n+1}_+$. The same is therefore true of
$\wt\phi\otimes\wt\phi$.

To sum up, $(\vv,[\phi])\mapsto[\phi,\,\vv\cdot\phi]$ defines a smooth
injective mapping
\[ \rr^{2n}\times Z_n\to Z_{n+1}.\] We leave the reader to
check that the differential of this mapping has full rank, so that its
image is open. A comparison with (\ref{twist1}) shows that we may
identify this image with $\pi^{-1}(\rr^{2n})$, in which case the
missing fibre $\pi^{-1}(\infty)$ consists of those points $[0,\psi]$
for which $\phi$ vanishes.
\end{proof}

We can now use the mapping
\begin{align*}
\begin{array}{c}
\Z{n+1}\setminus \Z n \to \rr^{2n},\qquad
[\phi,\,\vv\cdot\phi] \mapsto\vv
\end{array}
\end{align*}
to realize the twistor projection. We shall make it explicit in order
to parametrize $\Z{n+1}$ as a \emph{complex analytic} manifold.

To this aim, we first introduce quantities
\begin{align}
\label{Wi}
W_i = \xi_0 z^i - \suml_{k=1}^n \xi_{im}\bz^m,\qquad i=1,\ldots,n
\end{align}
This definition is an exact parallel of (\ref{etai}), and each $W_i$
is merely the contraction of $\eta_i$ with an arbitrary vector
$v\in\rr^{2n}$ expressed with coordinates $z^1,\ldots,\bz^n$. It
follows that the $W_i$ can be regarded as the components of the
Clifford product $\vv\cdot\bxi$ in the first summand $\Lambda^{1,0}$ of
(\ref{Delta-ext}). Next, we treat $W_i$ as a function of both $\bxi$
and $\vv$.

A crucial observation is that
\begin{align*}
dW_i=\eta_i+z^id\xi_0-\suml_{k=1}^n \bz^m d\xi_{im}.
\end{align*}
is a $(1,0)$-form relative to the complex structure on $\rr^{2n}$
defined by any point $J\in\Z n$ whose homogeneous coordinates in
$\cp^{N-1}$ start with $\xi_0$ and the $\xi_{ij}$. It follows that
each $W_i$ is a \emph{holomorphic} function on the twistor space
$(\sZ(\rr^{2n}),\sJ)$; this is because the value $\sJ$ at
$(J,\mathbf{x})$ is defined relative to the complex structure that $J$
itself induces on $\rr^{2n}$; recall (\ref{decree}).

More generally, and in parallel to (\ref{etaiq}), we define functions
\begin{align}
\label{Wdef}
W_{i_1\dots i_q} = \suml_{k=1}^q (-1)^{k-1} \xi_{i_1\dots\hat{i_k}\dots i_q}
z^{i_k} - \suml_{m=1}^n\xi_{i_1 \dots i_q m} \bz^m,\ \quad q\ge3\hbox{ odd}.
\end{align}
These are effectively the components of $\vv\cdot\bxi$ in
$\Lambda^{q,0}$ in \eqref{Delta-ext}, and the above considerations
apply. Representing a point of $\cc^n=(\rr^{2n},\bJ)$ by
$\bfz=(z^1,\ldots,z^n)$, we are now in a position to define a map
\begin{align}
\label{FFF}
\FFF:\Z n \times\rr^{2n} \rightarrow\Z{n+1} \subset\cp^{2N-1},
\end{align}
by 
\begin{align*}
([ \xi_0, \xi_{12},\ldots, \xi_{1\cdots n}],\, \bfz) \mapsto
[ \xi_0, \xi_{12}, \dots, \xi_{1\cdots n}, W_1,\ldots,W_{123},\ldots,
W_{2\cdots n}],\\[2pt]
([ \xi_0, \xi_{12},\ldots, \xi_{2\cdots n}],\,\bfz) \mapsto
[ \xi_0, \xi_{12}, \dots, \xi_{2\cdots n}, W_1,\ldots,W_{123},\ldots,
W_{1\cdots n}],
\end{align*}
according as $n$ is even or odd, respectively.

\begin{example}\label{6dim}
{\em When $n=3$, we are mapping $([\bxi],\bfz)$ to
\begin{align}
\label{xiW}
[\xi_0, \xi_{12}, \xi_{13}, \xi_{23}, W_1, W_2, W_3, W_{123}],
\end{align}
where 
\begin{align}
\begin{split}
\label{Wi6}
W_1 &= \xi_0z^1-\xi_{12}\bz^2-\xi_{13}\bz^3\\
W_2 &= \xi_0z^2-\xi_{23}\bz^3+\xi_{12}\bz^1\\
W_3 &= \xi_0z^3+\xi_{13}\bz^1+\xi_{23}\bz^2,
\end{split}
\end{align}
(cf.~(\ref{etai6})), and 
\begin{align}
\label{W123}
W_{123} = z^1 \xi_{23} - z^2 \xi_{13} + z^3 \xi_{12}. 
\end{align}
It follows that 
\begin{align}
\label{xidotW}
\xi_0W_{123}=\xi_{12}W_3-\xi_{13}W_2+\xi_{23}W_1,
\end{align}
which is a reincarnation of (\ref{1324}) in the twistor context.
Slightly different versions of this quadratic equation will recur
repeatedly in the remainder of this paper.}
\end{example}

In conclusion, 

\begin{theorem}
\label{BIHOLO}
  The map $\FFF$ is a biholomorphism from $\Z n\times \rr^{2n}$ to
  $\Z{n+1} \setminus\Z n$, where $\Z n \times \rr^{2n}$ has the
  complex structure $\sJ$, and $\Z{n+1}$ is a complex submanifold of
  $\cp^{2N-1}$.  The missing $\Z n$ is given by points with
  the first $N=2^{n-1}$ coordinates equal to zero, i.e., points of the
  form $[ 0, \dots, 0, W_1, \dots, W_{\cdots n}]$.
\end{theorem}

By adding the missing twistor fiber over the point at infinity, we
obtain
\begin{corollary}
The map $\FFF$ can be extended to a biholomorphism 
\begin{align*} 
\hat\FFF: \sZ(S^{2n}) \rightarrow\Z{n+1}.
\end{align*}
\end{corollary}
\noindent
The map $\hat\FFF$ is then another realization of the fibration
\eqref{twist1}.\smallbreak

At this juncture, as an application, we state and prove a general
integrability result. Although this is fairly well-known, and was
proved by the second author in \cite{Salamon1985}, it is readily
formulated in the language of this section.
\begin{proposition}
Let $J$ be an orthogonal almost complex structure defined on an open
set $\Omega \subset S^{2n}$.  Then $J$ is integrable if and only if
the graph $J(\Omega)$ is a holomorphic $n$-fold in $(\pi^{-1}(
\Omega),\sJ)$.
\end{proposition}
\begin{proof}
  The invariant nature of the statement of the proposition makes it
  sufficient for us to prove it locally. We may therefore assume that
  the space of $(1,0)$-forms for $J$ is generated by the 1-forms
  $\eta_i$ defined in (\ref{etai}) with $\xi_0=1$. In this way, $J$ is
  entirely determined by a skew-symmetric matrix $(\xi_{ij})$ of
  smooth functions on $\Omega$.

The graph of $J$ will be holomorphic if and only if the $\xi_{ij}$
depend holomorphically on the $W_s$-s, i.e.
\begin{align*}
0 = \frac{\pd\xi_{ij}}{\pd\bW_{\!s}} 
= \suml_l\displaystyle\frac{\pd z^l}{\pd\bW_{\!s}}\xi_{ij,l}
+ \suml_m\displaystyle\frac{\pd\bz^m}{\pd\bW_{\!s}}\xi_{ij,\bar m},
\qquad s=1,\ldots,n.
\end{align*}
We can use (\ref{Wi}) to find the Jacobian matrix
\begin{align*}
\frac{\pd(W_r,\bW_{\!s})}{\pd(z^l,\bz^m)} = 
\left(\begin{matrix}
I & -(\xi_{rm}) \\
-(\xi_{sl}) & I  \\
\end{matrix}
\right).
\end{align*} 
Inverting this produces (up to a determinant) the ``same'' matrix with
no minus signs. Thus $\pd z^l/\pd\bW_{\!s}=\xi_{sl}$ and the condition
is $\Xi_{ijk}=0$ where
\begin{align}
\label{cond}
\Xi_{ijs}=\xi_{ij,\bar s}+\suml_l \xi_{ij,l}\xi_{sl},
\end{align}
and the commas indicate ``Euclidean'' partial differentiation.

We next compute
\begin{align*}
d \eta_i = - (\xi_{ij,l} dz^l +\xi_{ij, \bar{l}} d\bz^l)\wedge d\bz^j.
\end{align*}
For integrability, we need $d\eta_i(v_r,v_s)=0$ for all $i,r,s$, where
the vectors
\begin{align}
\label{vdel}
v_r = \od_r - \suml_{k=1}^n \xi_{rk}\pd_k, 
\end{align}
span $T^{0,1}$, as in (\ref{vvv}). A computation shows that 
integrability of $J$ is then equivalent to the condition that
\begin{align*}
\Xi_{ijk}=\Xi_{ikj}.
\end{align*}
But since $\Xi_{ijk}$ is skew-symmetric $i,j$, we obtain 
\begin{align*}
\Xi_{ijk} = - \Xi_{jik} = - \Xi_{jki} = \Xi_{kji} = \Xi_{kij} =
-\Xi_{ikj} = -\Xi_{ijk},
\end{align*}
which implies that $\Xi_{ijk}=0$. (The parallel to the proof of the
fundamental theorem of Riemannian geometry arises from the fact that
(\ref{vdel}) can be viewed as a covariant derivative of the form
$\nabla_{\overline i}$.)  This completes the proof.\end{proof}

\begin{remark}{\em
  A direct corollary is the non-existence of a global OCS on $S^6$;
  see \cite{LeBrunS6}. Our work does not shed further light on the
  question of whether $S^6$ admits a complex structure, although it is
  conceivable that generalizations of the twistor approach might be
  relevant to this problem. For some intriguing properties of a
  hypothetical complex structure on $S^6$, we refer the reader to
  \cite{HKP}.}
\end{remark}

\section{Warped product structures}
\label{warped}

Consider $\rr^{2n}$ with coordinates $(z^1,\dots, z^n)$, and consider a
smooth orthogonal almost complex structure of the form
\begin{align}
\label{wpo}
J = J_1 \oplus J_0,
\end{align}
where $J_1=J_1(z^1, \bz^1, \dots z^n, \bz^n)$ is an OCS (and so an
\emph{integrable} complex structure) defined on
\begin{align*}
\rr^{2n-2} = \{ z^n = \const\},
\end{align*}
and $J_0$ is the standard OCS on the complementary $\rr^2$ spanned by
the real and imaginary parts $x^n$ and $y^n$ of $z^n$. If $J$ is
itself integrable, we shall call it a \emph{\wp\ orthogonal complex
  structure}. Because there is only one OCS on $\rr^2$ up to sign, an
equivalent way of saying this is the following:

\begin{definition}
\label{wpdef}
{\em    A \emph{\wp\ OCS} on $\rr^{2n}$ is an orthogonal complex structure
  $J$ on $\rr^{2n}$ preserving an orthogonal decomposition
  $\rr^{2n}=\rr^{2n-2}\oplus\rr^2$ pointwise.}
\end{definition}

\noindent We shall tacitly assume that $J_0,J_1$ (and so $J$) are
consistent with fixed orientations of the respective Euclidean spaces.
Note that a \wp\ OCS is constant in dimension 4 since there is only
one oriented OCS on $\rr^2$.\smallbreak

Given a \wp\ OCS $J$, the orthogonal projection
\begin{align*}
\pi\colon(\rr^{2n},J)\to(\rr^2,J_0)
\end{align*}
is a \emph{holomorphic} mapping between Hermitian manifolds since
$\pi_*\circ J=J_0\circ\pi_*$. In particular, $J_1$ defines an
integrable complex structure on each fibre
$\pi^{-1}(v)\cong\rr^{2n-2}$ with $v\in\rr^2$. In this way, $\pi$
determines a complex analytic \emph{deformation} of
$(\pi^{-1}(0),J_1)$ in the orthogonal category. If we want the base
parametrizing the deformations to be a Euclidean space
$\rr^{2m}\cong\cc^m$ with a constant OCS, there is no restriction in
assuming that $m=1$ as we do, since the general case falls within the
scope of Definition~\ref{wpdef}.

\begin{proposition}
\label{wpint}
Let $J$ be an almost complex structure of the form \eqref{wpo} with
each $J_1$ an OCS. Then $J$ is integrable if and only if the mapping
\begin{align}
\label{mapuv}
(\rr^2,J_0)\to\Z{n-1},\qquad v\mapsto J_1(u,v)
\end{align}
is holomorphic for each fixed $u\in\rr^{2n-2}$. 
\end{proposition}

\begin{proof}
  The value of $J$ at each point of $\rr^{2n}$ corresponds to a pure
  even spinor as in (\ref{bxi}). We can ensure that this has the form
  (\ref{wpo}) by requiring that all components that contain an index
  $n$ vanish, as stated in Proposition~\ref{gendim}. There remain
  $2^{n-1}$ (potentially non-zero) components
  $\xi_{ij}$, which rightly define an OCS on
  $\rr^{2n-2}$. Let us assume that $\xi_0\ne0$, and then scale so that
  $\xi_0\equiv1$. Setting $s=n$ in the integrability equations
  (\ref{cond}), we have
\begin{align*}
\xi_{ij,\bar n} + \suml_l\xi_{ij,l}\xi_{nl} = 0.
\end{align*}
Thus $\xi_{ij,\bar n}=0$, which says that $\xi_{ij}$ is holomorphic in
the $z^n$ coordinate, which amounts to the holomorphicity of
\eqref{mapuv} for fixed $u$. The remaining integrability equations
just say that for each $z^n$ fixed, $J_1$ is integrable as an OCS in
$\rr^{2n-2}$, which is what we are in any case assuming.
\end{proof}

The importance of the preceding proposition is that it enables one to
construct warped product structures with relative ease. The simplest
way of doing this is to choose a holomorphic function
$f\colon\cc\to\Z{n-1}$ and then define $J_1(z^n)$ to be the constant
OCS on $\pi^{-1}(z)=\rr^{2n-2}$ corresponding to $f(z^n)$. We shall do
this in some examples below, but first we place our definitions in a
more general context.\smallbreak

Let $(M,J,g)$ be a Hermitian manifold of real dimension $2n$. This
means that $J$ is a complex structure defining a transformation at
each point that is orthogonal relative to the Riemannian metric $g$,
and there is an associated 2-form $\omega$ by
\begin{align*}
\omega(X,Y)=g(JX,Y).
\end{align*}
The Hermitian manifold is \emph{locally conformally K\"ahler} if there
exists a positive function $\lambda$ in a neighborhood of each point
such that $d(\lambda\omega)=0$. A complementary condition is that $J$
be \emph{cosymplectic}, meaning that
\begin{align}
\label{*omega}
*\kern1pt\omega=\frac1{n!}\omega^{n-1}
\end{align}
is closed, $*$ being the Hodge operator. This concept is only useful
if $n\ge3$ since ``cosymplectic'' is equivalent to ``K\"ahler'' on a
Hermitian manifold of real dimension 4. If general, if $M$ is both
locally conformally K\"ahler and cosymplectic then
\begin{align}
\label{lckco}
0= d(\lambda\omega)\wedge\omega^{n-2}=d\lambda\wedge\omega^{n-1},
\end{align}
so $\lambda$ is constant and $M$ is K\"ahler.

These properties are easily assessed for the structures of
Definition~\ref{wpdef}. Take $M$ to be $\rr^{2n}$ with the Euclidean
metric $g=\geu$ and let $J$ be a \wp\ orthogonal complex structure.
The K\"ahler condition $d\omega=0$ implies that $\nabla J=0$ where
$\nabla$ denotes the Levi~Civita connection for $\geu$. Thus,
$\nabla_XJ=0$ for all $X$, where $\nabla_X$ is the Euclidean
directional derivative, and any K\"ahler OCS is necessarily
\emph{constant} on $\rr^{2n}$. Modulo orientation, such a $J$ defines
a \emph{horizontal} section of the twistor space $\sZ$ in \eqref{sZ},
and is effectively an element of $\Z n$. The next result shows that,
of the weaker conditions considered above, the cosymplectic one is
more relevant to the warped product situation.

\begin{proposition}\label{cosymp}
  Let $J$ be a \wp\ OCS on $\rr^{2n}$ with $n \ge 3$.
\begin{enumerate}
\item[(i)] $J$ is locally conformally K\"ahler if and only if it is
  constant on $\rr^{2n}$.
\item[(ii)] $J$ is cosymplectic if and only if $J_1$ is cosymplectic
  on each $\rr^{2n-2}$.
\end{enumerate}
\end{proposition}

\begin{proof} We first prove (ii). Write
$\omega = \omega_1 + \omega_0$, where $\omega_0=dx^n\wedge dy^n$ is
  closed, so
\begin{align}
\omega^{n-1}= \omega_1^{n-1} + (n-1)\omega_0\wedge\omega_1^{n-2}.
\end{align}
The first term on the right is a volume form on $\rr^{2n-2}$, and so
globally constant. Equation \eqref{*omega} implies that
\begin{align}
n!\,d(*\omega) = 
(n-1)(n-2)\omega_0\wedge(d\omega_1\wedge\omega_1^{n-3}).
\end{align}
The vanishing of the right-hand side is equivalent to asserting that
$J_1$ is cosymplectic.

In (i) the ``if'' statement is trivial. So suppose that $J$ is locally
conformally K\"ahler. Since $H^1(\rr^{2n})=0$, the qualification
``locally'' can be dropped and we may suppose that
$d(\lambda\omega)=0$, in which the conformal factor $\lambda$ is
defined globally. It follows that $d\lambda\wedge\omega_0=0$ since
this is the only term in the exterior derivative of $\lambda\omega$
divisible by $dx^n\wedge dy^n$. Hence $\lambda=\lambda(x^n,y^n)$ is
constant on each $\rr^{2n-2}$. It follows that $J_1$ is K\"ahler and
constant on each $\rr^{2n-2}$. By (ii), $J$ is cosymplectic and from
the argument \eqref{lckco}, we conclude that $J$ is also K\"ahler and
constant on $\rr^{2n}$.
\end{proof}

\begin{remark}
{\em
When $n=2$, both statements are trivial since, as previously
remarked, a \wp\ OCS must be constant in this case. When $n=3$, $J$ is
always cosymplectic since the 4-dimensional Liouville theorem in
\cite{SalamonViaclovsky} implies that $J_1$ must be constant.  Part
(ii) is motivated by an example in \cite{BairdWood}.

Part (i) is also a consequence of a more general but well-known
result: a conformally flat K\"ahler metric in dimension $2n\ge6$ is
flat \cite[Theorem 4.1]{YanoMogi}, \cite[Proposition 2.68]{Besse}. 
This implies that a conformally K\"ahler $J$ will necessarily be K\"ahler 
relative to $\geu$, and therefore constant by the discussion preceding 
Proposition~\ref{cosymp}.  
By contrast, there \emph{do} exist non-flat conformally flat K\"ahler
metrics in dimension 4 including, for example, a complete product one
on $S^2\times H^2\cong\rr^4\setminus\rr$ that plays a key role in the
classification of OCSes in dimension 4 \cite{Pontecorvo, SalamonViaclovsky}.}
\end{remark}

We next look at some non-trivial examples of warped product
structures. In dimension $6$, a \wps is obtained from the spinor
\begin{align}
\label{WPS6}
\phi=[\xi_0, \xi_{12}, \xi_{13}, \xi_{23}]=[\xi_0, \xi_{12}, 0, 0].
\end{align}
This gives an OCS on $\rr^4=\{z^3=\const\}$ and, as remarked above,
this must be a \emph{constant} orthogonal complex structure. Assuming
$\xi_0\ne0$, Proposition~\ref{wpint} tells us that a \wp\ OCS arises
from
\begin{align}
\label{phie}
\phi = [1,\xi_{12}(z^3),\,0,\,0],
\end{align}
with $\xi_{12}(z^3)$ a \emph{holomorphic} function of $z^3$.

\begin{remark}
\label{4dt}
{\em If we simply take $\xi_{12}(z^3)=z^3$ in \eqref{phie}, then we
  can identify
\begin{align*}
\rr^6 = \rr^4 \times \>\cc\> \subset \rr^4\times\cp^1=\sZ
\end{align*}
as a complex submanifold of the twistor space of $\rr^4$ (see
\eqref{sZ}). In this case, by rescaling vertically, \eqref{wpo}
extends to a Hermitian structure on $\sZ$, the fiber $\cp^1$ being a
conformal compactification of $\cc$. If we also rescale the resulting
metric horizontally, it extends further to the standard Hermitian
structure of $\Z3=\cp^3$.  

Combining a similar argument with
Theorem~\ref{BIHOLO} in higher dimensions, we see that the choice of a
rational curve $\cp^1\subset\Z{n-1}$ (and a marked point to remove)
exhibits a \wp\ $(\rr^{2n},J)$ as a complex submanifold of $\Z n$.
The \wp\ construction can be viewed as a generalization of these
examples.}
\end{remark}

In dimension $8$, a pure spinor class
\begin{align*}
\phi=[ \xi_0, \xi_{12}, \xi_{13}, \xi_{14}, \xi_{23}, \xi_{24}, 
\xi_{34}, \xi_{1234}],
\end{align*}
must lie on the quadric (\ref{1324}). A \wp\ arises from
\begin{align*}
\phi=[ \xi_0, \xi_{12}, \xi_{13}, 0, \xi_{23}, 0, 0, 0]
\end{align*}
Restricted to a hyperplane $\{z^4 = \const\}$, this will give an OCS
on $\rr^6$, with each component depending holomorphically on the
coordinate $z^4$.

A special case will occur when the induced OCSes on the hyperplanes
$z^4=\const$ are \wp s in the same direction, that is
$\xi_{13}=\xi_{23} =0$, and $\xi_{12}=\xi_{12}(z^3)$. This will
look like
\begin{align*}
[ \xi_0(z^3, z^4), \xi_{12}(z^3, z^4), 0, 0, 0, 0, 0, 0].
\end{align*}
Writing $\rr^8 = \rr^4 \oplus \rr^4$,
these are \wp s of the form 
\begin{align*}
J = J_1\oplus\widetilde{J_0}, 
\end{align*}
where $J_1$ is a constant OCS on each 4-dimensional hyperplane for
which both $z^3,z^4$ are constant and $\widetilde{J_0}$ is the
standard OCS on each complementary $\rr^4$. By
Proposition~\ref{wpint}, it now suffices to take the map
$(\rr^4,\widetilde{J_0})\to\cp^1$ determined by $J_1$ to be
holomorphic.

\begin{question}
If $J$ is an finite energy OCS on $\rr^{2n}$, then is $\pm J$
conformally equivalent to a \wp\ OCS of the form {\rm(\ref{wpo})}?
\end{question}
In dimension $4$, such an OCS is constant
\cite{WoodIJM,SalamonViaclovsky}, and no finite energy assumption is
necessary.  The main result of this paper is that the answer is
\emph{yes} in dimension $6$.  There does not seem to be any other
obvious way to manufacture examples in higher dimensions other than
the \wp\ construction, which would lead one to conjecture
the answer might be \emph{yes} in general. However, the twistor spaces
become increasingly more complicated as the dimension grows, giving
more room for the possibility of complicated subvarieties which could
be graphs over $\rr^{2n}$.

For example, if $J_{2n}$ is an OCS defined globally on $\rr^{2n}$,
then the graph of $J_{2n}$ lies in $\Z{n+1}\setminus\Z n$, where the
$\Z n$ is the twistor fiber over the point at infinity.  The closure
will add some subvariety of $\Z n$ of complex dimension $n-1$. But
since $\Z n$ is the twistor space of $\rr^{2n-2}$, this will
correspond to some OCS $J_{2n-2}$ on some subset of $\rr^{2n-2}$. It
can happen that $J_{2n}$ is a \wp\ involving a deformation of
$J_{2n-2}$, but it is possible that there are examples in higher
dimensions where this fails.
\subsection{Warped product structures on tori}
\label{warptori}
In this subsection, we discuss the construction of the examples in
Theorem~\ref{torusex}. 

As mentioned in the Introduction, if $J$ is an OCS on a flat $4$-torus
$(T^4,g_4)$, then $J$ lifts to a constant OCS on $(\rr^4,\geu)$, where
$\geu$ is the Euclidean metric.  Consequently, the OCSes on
$(T^4,g_4)$ compatible with a fixed orientation are parametrized by
$\Z 2 = \cp^1$. Take an elliptic curve $(T^2,J_2)$ with a compatible
flat metric $g_2$, and consider the product torus $(T^6, g_6) = (T^4
\times T^2, g_4\oplus g_2)$.

We endow $T_6$ with a \wp\ OCS
\begin{align*}
J_6 = J_4\oplus J_2,
\end{align*}
where $J_4$ is determined by a holomorphic map $f: (T^2,
J_2) \rightarrow \Z2\cong\cp^1$, which is the same as a meromorphic
function on $\cc$ invariant under the corresponding lattice.  Such
functions are exactly quotients of translated $\sigma$-functions, see
\cite[Chapter 7]{Ahlfors}, and are non-algebraic if not constant.
Thus, if non-constant, these structures must have infinite energy when
lifted to $\rr^6$, since Bishop's Theorem says that finite energy
implies algebraic, see Section \ref{finalsec}.

In dimension $8$, we can perform the following construction. 
We can consider $(T^8, g_8) = (T^6 \times T^2, g_6 \oplus g_2)$, 
and endow this with a \wps 
\begin{align*}
J_8 = J_6\oplus J_2.
\end{align*}
If we take $J_6$ to be a constant OCS on $T^6$, then it is determined
by a holomorphic mapping $( T^2,J_2) \rightarrow\Z3\cong\cp^3$.
However, this need not be so; for example, $J_6$ could itself be a
\wp\ OCS on $T^6$. In this case, Proposition~\ref{wpint} tells us that
$J_8$ arises from a holomorphic mapping from $(T^2,J_2)$ into the
space of meromorphic functions of fixed degree on another elliptic
curve.

In a similar fashion, one can construct increasingly complicated \wp\
structures on tori in all higher even dimensions.

\section{Asymptotically constant structures}
\label{finding}
In this section we prove Theorem \ref{t3}. We first make some remarks
in the case of dimension $6$, and then give the general argument. 

Let $J$ be an OCS defined globally on $\rr^6$ with the correct
orientation. Then $J(\rr^6)$ is a smooth variety inside $\Z4 =
\Q^6$. Using the notation of Example~\ref{6dim}, the graph $J(\rr^6)$
is given by the expression (\ref{xiW}). Its closure is found by taking
all limits of sequences $(z^1_j, z^2_j, z^3_j)$, where at least one of
the sequences $z^i_j$ approaches infinity as $j\to\infty$.  This will
of course only add points of the form
\begin{align*}
[0,0,0,0, W_1, W_2, W_3, W_{123}], 
\end{align*}
which are points in $\cp^3_{\infty}=\pi^{-1}(\infty)$.

We next consider some special cases.  First, if $\xi_{12}, \xi_{13},
\xi_{23}$ are constant, by a conformal transformation, we may assume
that $\xi_{12} = \xi_{13} = \xi_{23} = 0$. The graph is then simply
$[\xi_0, 0, 0, 0, \xi_0 z^1, \xi_0 z^2, \xi_0 z^3, 0]$. Clearly, the closure adds the
$\cp^2$ given by
\begin{align*}
[0,0,0,0,W_1,W_2,W_3,0],
\end{align*}
whereas the $\cp^3$ given by
\begin{align}
\label{ccp3}
[\xi_0,0,0,0,W_1, W_2, W_3, 0]
\end{align}
is exactly the closure of the graph of the constant OCS $\bJ$ on
$\rr^6$ corresponding to $\phi=[1,0,0,0]$.
\begin{proposition}
\label{cstp6}
If $J$ is an OCS on $\rr^6$ which is asymptotic to the one defined by
$\phi=[1,0,0,0]$, the closure of the graph of $J$ adds the $\cp^2$
given by
\begin{align}
\label{1cp}
[0,0,0,0,W_1,W_2,W_3,0].
\end{align}
Moreover the closure of the graph of $J$ is homeomorphic to $\cp^3$.
\end{proposition}
\begin{proof}
The asymptotically constant condition means that 
the distance (in a suitable metric) between the 
graph of $J$ and the graph of the constant 
OCS goes to zero as $z \rightarrow \infty$, so 
the closure must add the same points. 
There is then an obvious homeomorphism between 
the graph of $J$ and the $\cp^3$ given in \eqref{ccp3}. 
\end{proof}
The same idea also works in higher dimensions.
\begin{proposition}
\label{pr1}
Let $J$ be a constant OCS on $\rr^{2n}$. Then $\pm J$ is isometrically
equivalent to the OCS $\bJ$ corresponding to $[1, 0, \dots, 0] \in
Z_n^+$.  Furthermore, the closure of the graph of $\bJ$ adds the
$\cp^{n-1}$ in the fiber over infinity given by
\begin{align}
\label{ccpn}
[W_1, \dots, W_n, 0,\dots, 0],
\end{align}
that is, all $W_* = 0$ if the multi-index $*$ is of length 
$3$ or greater. 
\end{proposition}
\begin{proof}
The orthogonal group $\SO(2n)$ acts transitively on $Z_n^+$
\cite{Cartan}, so we can simply rotate to arrange
that $J = \bJ$. 
Next, consider the $\cp^n \subset Z_{n+1}^+$, call it $P$, defined 
by  $\xi_* = 0$ for any multi-index $*$ of length $2$ 
or greater, and $W_* = 0$ if the multi-index $*$ is of length 
$3$ or greater. That is, in the $[\xi , W]$ coordinates on 
$Z_{n+1}^+$, this is given by 
\begin{align*}
[(\xi_0, 0, \dots, 0), (W_1, \dots W_n), 0, \dots, 0].
\end{align*}
We claim that $P$ is the closure of the 
graph of $\bJ$. To see this, using our twistor 
coordinates, the graph of 
$\bJ$ over $\rr^{2n}$ is given by
\begin{align*}
[(\xi_0, 0, \dots, 0), (\xi_0 z^1, \dots, \xi_0 z^n, 0, \dots, 0)].
\end{align*}
To find the closure, we take all possible limits 
as $z \rightarrow \infty$, and this clearly adds all points in \eqref{ccpn}.
\end{proof}

Without loss of generality we may therefore assume that 
$J$ is asymptotic to $\bJ$.
\begin{proposition}
\label{cstp}
If $J$ is an OCS on $\rr^{2n}$ which is asymptotic to $\bJ$, 
then the closure of the graph of $J$ adds
the same $\cp^{n-1}$ to the fiber over infinity as does 
$\bJ$. Moreover, the closure of the graph of $J$ is homeomorphic to $\cp^n$.
\end{proposition}
\begin{proof}
Exactly as in the six-dimensional case, the asymptotically constant 
condition means that the distance (in a suitable metric) between the 
graph of $J$ and the graph of the constant 
OCS goes to zero as $z \rightarrow \infty$, so 
the closure must add the same points as does $\bJ$. 
Since both the graphs of $J$ and $\bJ$ hit every other 
twistor fiber in a single point, there is then an obvious 
homeomorphism between the graph of $J$ and the $\cp^{n}$ given in \eqref{ccpn}
corresponding to $\bJ$. 
\end{proof}

\begin{proposition} 
If $J$ is an OCS on $\rr^{2n}$ which is asymptotic to $\bJ$, 
then the closure of the graph of $J$ is a linear $\cp^n$. 
\end{proposition}
\begin{proof}
From the previous proposition, the assumption implies that 
$\X = \overline{J(\rr^{2n})}$ is homeomorphic
to $\cp^n$, which is contained inside the twistor space
$\Z{n+1}\subset\cp^{2N-1}$ where $N=2^{n-1}$. Since $\X$ is
homeomorphic to a manifold, it satisfies Poincar\'e duality. Also,
taking the closure adds $\X_0 = \cp^{n-1}$ inside the fiber at
infinity, so by the Thullen--Remmert--Stein Theorem, $\X$ is
necessarily a variety  \cite{Thullen, RemmertStein}.

We denote by $H$ the hyperplane section class on $\cp^{2N-1}$, and
also its pullbacks to subvarieties. Then the cup product $H^{n-1} \cup
\X_0$ equals $1$ on $X$. Since $H^2(\X)$ has rank~$1$ with generator
the (Poincar\'e dual of) $\X_0$, we have $H=k\X_0$ in cohomology,
which implies that $k^{n-1}\X_0^n = 1$.  In view of the integrality of
intersection numbers, $k=1$ so $H=\X_0$.  This means that $H^n= 1$ on
$\X$, so the degree of $\X$ is $1$, which implies that $\X$ is a
linear $\cp^n$ in $\cp^{2N-1}$; see \cite[page~174]{GriffithsHarris}.
\end{proof}

\begin{proposition}
If $J$ is an OCS on $\rr^{2n}$ which is asymptotic to $\bJ$, 
then the closure of the graph of $J$ 
is the same linear $\cp^n$ as that which corresponds to $\bJ$.
\end{proposition}
\begin{proof}
From the previous proposition, $X$ is a linear 
$\cp^{n}\subset \Z{n+1}\subset\cp^{2N-1}$ where $N=2^{n-1}$.
Therefore, there are exactly $2N - 1 - n$ linear equations 
defining $X$, which we write as 
\begin{align*}
{\bf{a}}_j \cdot \bxi + {\bf{b}}_j 
\cdot {\bf{W}} & = 0,\quad j = 1, 2,  \dots,  2^n - 1 - n
\end{align*}
(with slight abuse of notation). If we restrict these equations to
the fiber over infinity given by $\xi_* = 0$, we have
\begin{align*}
{\bf{b}}_j\cdot{\bf{W}} & = 0,\quad j = 1,2,\dots, 2^n - 1 - n.
\end{align*}
However, we know that these equations must define the $\cp^{n-1}$ from
$\eqref{ccpn}$, which is the condition that $W_* = 0$ if the
multi-index $*$ is of length $3$ or greater. Consequently, by taking
linear combinations, we may regroup the defining equations as
\begin{align}
\label{We}
{\bf{a}}_j'\cdot \bxi + W_{*_j} & = 0,\quad j = 1,2,\dots,2^{n-1}-n,
\end{align}
where $*_j$ is a multi-index of length $3$ or greater (and $j$ now
labels these multi-indices), together with
\begin{align}
\label{We2}
{\bf{a}}_j''\cdot \bxi = 0,\quad j = 1,2,\dots, 2^{n-1}-1.
\end{align}
All these equations must be linearly independent (else our equations
would define a subspace $\cp^l$, $l > n$). The collection \eqref{We2}
therefore specifies a single point in the fiber over the origin. When
we restrict \eqref{We} to the origin (so $W_{*_j}=0$) we cannot have
$\bxi=0$. Consequently, so we must be able to use \eqref{We2} to rid 
\eqref{We} of all $\xi$ terms.
We have therefore reduced the equations to the form
\eqref{We2} and
\begin{align}
\label{Wee}
W_{*_j} &= 0,\quad *_j \mbox{ is a multi-index of length $3$ or
  greater},
\end{align}
Recall from \eqref{Wdef} that if $*$ is a
multi-index of length $q\ge3$ then
\begin{align}
\label{Wq}
W_* = \xi_-\cdot z + \xi_+\cdot \bar{z} 
\end{align}
where $-,+$ represent multi-indices of lengths $q-1,q+1$.  But $\xi$
is determined by \eqref{We2}, and the only way that \eqref{Wq} can 
vanish is if all of the coefficients of $z$ and $\bar{z}$ 
are zero. From \eqref{Wee}, this is true for all odd multi-indices of length $3$ or
greater, which shows that $\xi_*= 0$ for all even multi-indices of
length $2$ or greater. Therefore, the $\cp^n$ is question is the same
as that corresponding to $\bJ$.
\end{proof}
This completes the proof of Theorem \ref{t3}.
\begin{remark}{\em
We shall show in Section~\ref{finalsec} that the closure of $J(\rr^6)$
for any globally defined $J$ adds a $\cp^2$ in the twistor fiber over
infinity.  However, $\X = \overline{J(\rr^6)}$ can have singularities,
and will not necessarily be homeomorphic to a manifold.  The above
proof will not work since $\X$ will then not necessarily satisfy
Poincar\'e duality.}
\end{remark}
\section{The twistor fibration to $S^6$}
\label{twistorprojection}

In this section, we shall provide an explicit matrix representation of
the Clifford multiplication
\begin{align}
\label{Cliffmult}
   \rr^6\otimes\Delta_\pm\to\Delta_\mp.
\end{align}
We shall then use this, in accordance with Theorem~\ref{Clifford}, to
describe elements of the twistor space $Q^6$ of $S^6$.  The results of
this section can also be interpreted in terms of Cayley numbers, but
the approach we adopt will provide an effective description of the
action of the conformal group.

To emphasize the symmetry underlying the definitions of the functions
$W_1,W_2,W_3$ and $W_{123}$, we introduce the following 4-vectors for
exclusive use in this section:
\begin{align}
\begin{split}
\label{4vec}
\bxi &= (\xi_0,\,\xi_{12},-\xi_{13},\,\xi_{23})\in\Delta_+\\[3pt]
\bfW &= (-W_{123},W_3,W_2,W_1)\in\Delta_-.
\end{split}
\end{align}
They will be regarded as rows or columns, according to context. The
choice of order and signs here represents a compromise between our
previous ordering and a suitable canonical form for the matrix that
follows.

The four equations (\ref{Wi6}),\,(\ref{W123}) can now be combined into
the form $\bfW=\Mz\bxi$, where
\begin{align}
\label{Mmat}
\Mz =
\left(
\begin{matrix}
0    &   -z^3 &   -z^2 &   -z^1 \\
z^3  &  0     &  -\bz^1 & \bz^2 \\
z^2  & \bz^1 &    0   &  -\bz^3 \\
z^1  & -\bz^2 & \bz^3 &    0   \\
\end{matrix}
\right)
\end{align}
parametrizes a point in the Euclidean space $\cc^3=\rr^6$.  To
emphasize this, we let $\sE$ denote the set of all such matrices
$M(\bfz)$ with $\bfz=(z^1,z^2,z^3)\in\cc^3$. Note that $\sE$ is a
\emph{linear} subspace of the Lie algebra $\mathfrak{so}(4,\cc)$ of
skew-symmetric complex $4\times4$ matrices.

With the adjusted conventions \eqref{4vec}, the biholomorphism
$\FFF:\Z3 \times \rr^6\to\Z4 \setminus\Z3$ defined by \eqref{FFF} is
given by mapping $([\bxi],\,\bfz)$ to
\begin{align}
  \label{coordcp7} 
[\bxi,\,\Mz\bxi] =
  [\xi_0,\,\xi_{12},-\xi_{13},\xi_{23},\,-W_{123},W_3,W_2,W_1],
\end{align}
rather than (\ref{xiW}). The equation (\ref{xidotW}) characterizing
$\Z4$ can be neatly written
\begin{align}
\label{dotproduct}
\bxi\cdot\bfW=0,
\end{align}
relative to the standard complex bilinear pairing, and the fact that
$[\bxi,\,\bfW]\in\Z4$ is now a consequence of the skew symmetry of
$\Mz$.

The link with Clifford algebras is provided by the following result,
whose proof is a direct calculation.

\begin{lemma} Let $\bfy,\bfz\in\cc^3$. Then
\label{Clifford2}
\begin{align}
\label{zw+zw}
\overline{M(\bfy)}M(\bfz) + \overline{M(\bfz)}M(\bfy) =
-2\,\mathfrak{Re}\!\left<\bfy,\bfz\right>\!I,
\end{align}
where $I$ is the identity matrix, and
$\left<\bfy,\bfz\right>=\sum_{i=1}^3\!y^i\overline z^i$.
\end{lemma}

It will be convenient to denote by $\sE^*$ the subset
$\sE\setminus\{0\}$, and set
\begin{align*}
\hat\sE=\{M(\bfz)\in\sE^*:\|\bfz\|=1\},
\end{align*}
where $\|\bfz\|^2=\sum_{i=1}^3\!|z^i|^2$.  If $M=\Mz\in\hat\sE$ then
its columns are orthonormal in the Hermitian sense, and $M\in U(4)$.
Indeed, the lemma implies that $M\overline M\tp=\|\bfz\|^2I$, and
a direct calculation confirms that
\begin{align}
\label{detM} 
\det\Mz = \|\bfz\|^4. 
\end{align}
The next result establishes a curious link with the way in which
linear OCSes are themselves represented by matrices via (\ref{Pf}).

\begin{lemma}
$\hat\sE=\{M\in \SU(4)\cap\mathfrak{so}(4,\cc):\Pf M=1\}$.
\end{lemma}

\begin{proof} It is already clear that (\ref{Mmat})
  belongs to both $\SU(4)$ and $\mathfrak{so}(4,\cc)$, and a standard
  formula for the Pfaffian shows that $\Pf\Mz=1$. Suppose that
  $M\in \SU(4)\cap\mathfrak{so}(4,\cc)$.  Take the first column of $M$
  as in (\ref{Mmat}).  The second column must then coincide with that
  of (\ref{Mmat}), except that $\overline z^1$ and $-\overline z^2$ are
  possibility mutliplied by a complex number $\lambda$ of modulus 1.
  Analogous statements hold for columns 3 and 4 with the \emph{same}
  $\lambda$ that must satisfy $\lambda^2=1$. The only choice is to
  change all signs in the lower $3\times3$ block, and this reverses
  the sign of the Pfaffian.
\end{proof}

\noindent To sum up, $\sE$ is a cone over one component of the
intersection $\SU(4)\cap\mathfrak{so}(4,\cc)$.\smallbreak

Lemma~\ref{Clifford2} tells us that Clifford multiplication by $\bfz$
in (\ref{Cliffmult}) is represented by $\Mz$ on $\Delta_+$ and
by $\overline{\Mz}$ on $\Delta_-$. In this light, the next
result is a special case of Theorem~\ref{Clifford}, and consolidates
various arguments in the previous section.

\begin{theorem}
  If $[\bxi,\bfW]\in\Q^6$ and $\bxi\ne0$, then $\bfW = M(\bfz)\bxi$
  for some \emph{unique} $\bfz\in\cc^3$. The twistor projection
$\pi:\Q^6 \setminus\cp^3_\infty\to\rr^6$ is given by
$\pi( [\bxi,\,\bfW])=\bfz$, where
\begin{align}
\begin{split}
\label{zzz}
z^1 &= |\bxi|^{-2}\left(\bar{\xi}_0 W_1+\bar{\xi}_{23}W_{123} +
\xi_{13} \bar{W}_3 + \xi_{12} \bar{W}_2\right)\\ z^2 &= |\bxi|^{-2}
\left(\bar{\xi}_0 W_2 - \bar{\xi}_{13} W_{123} - \xi_{12} \bar{W}_1 +
\xi_{23} \bar{W}_3 \right)\\ z^3 &= |\bxi|^{-2} \left(\bar{\xi}_0 W_3+
\bar{\xi}_{12} W_{123} - \xi_{23} \bar{W}_2 - \xi_{13}
\bar{W}_1\right).
\end{split}
\end{align}
\end{theorem}

\begin{proof} Uniqueness essentially follows from Lemma~\ref{Clifford2}.
More directly, if $\Mz\bxi=M(\bfz')\bxi$, then
$(\Mz-M(\bfz'))\bxi=0$. Assuming $\bxi\ne0$, we obtain
\begin{align}
\label{!ness} 
\det(\Mz-M(\bfz)')=\|\bfz-\bfz'\|^4,
\end{align} 
from \eqref{detM}.

We have already observed that $[\bxi,\Mz\bxi]\in\Q^6$ provided
$\bxi\ne0$, and it induces an injective mapping
$f:\cp^3\times\sE\to\Q^6$. As in the proof of Theorem~\ref{Clifford},
we may conclude that the closure of the image of $f$ is obtained by
adding a copy of $\cp^3$ corresponding to points $[0,\bfW]$ generating
the fiber $\pi^{-1}(\infty)=\cp^3_{\infty}$. In any case, given
$[\bxi,\,\bfW]\in\Q^6$ with $\bxi\ne0$, the existence of $\Mz$ is now
guaranteed.

To establish the first equation of (\ref{zzz}), we proceed as follows.
Multiplying the equations in (\ref{Wi6}) and (\ref{W123}) by the
appropriate coefficients, we obtain
\begin{align*}
\begin{split}
\xi_{12} \bar{W}_2 &= \bar{\xi}_0 \xi_{12} \bz^2 
+ |{\xi}_{12}|^2 z^1 - \xi_{12} \bar{\xi}_{23} z^3\\
\xi_{13} \bar{W}_3 &= \bar{\xi}_0 \xi_{13} \bz^3 
+ |{\xi}_{13}|^2 z^1 + \xi_{13} \bar{\xi}_{23} z^2\\
\bar{\xi}_0 W_1 &= |\xi_0|^2 z^1 
- \bar{\xi}_0 \xi_{12} \bz^2 - \bar{\xi_0} \xi_{13} \bz^3\\
\bar{\xi}_{23} W_{123} &= |\xi_{23}|^2 z^1 
- \xi_{13} \bar{\xi}_{23} z^2 + \xi_{12} \bar{\xi}_{23} z^3. 
\end{split}
\end{align*}
Adding these four equations gives the required result:
\begin{align*} 
\xi_{12} \bar{W}_2 + \xi_{13} \bar{W}_3 + \bar{\xi}_0 W_1
+ \bar{\xi}_{23} W_{123} = |\bxi|^2 z^1.
\end{align*}
Given the cyclic symmetry in the components of $\bxi$ and
$\bfW$ in (\ref{zzz}), the second and third equations must also hold.

The fact that the $\bfz$ defines the twistor projection is a
consequence of the theory developed in the previous section. We leave
the reader to double check that, having defined $\bfz=(z^1,z^2,z^3)$
by (\ref{zzz}), it is indeed true that $\bfW=\Mz\bxi$.
\end{proof}

\begin{remark}{\em The component $\sE^-$ of matrices in
  $\rr^+\times(\SU(4)\cap\mathfrak{so}(4,\cc))$ with negative Pfaffian 
  parametrizes a different set of vertical $\cp^3$-s described in
  Subsection~\ref{linears}. If $L(\bfy)\in\sE^-$ is the matrix with
  first row $(0,-y^3,-y^2,-y^1)$
  then \[[\bzt,\,L(\bfy)\bzt]=[\bxi,\,M(\bfz)\bxi]\ \Rightarrow\ 
  (L(\bfy)-\Mz)\bxi=0,\] which implies that
\[ 0=\det(L(\bfy)-\Mz)=(\|\bfy\|^2-\|\bfz\|^2) + 
2i\,\mathfrak{Im}\left<\bfy,\bfz\right>.\] If we fix a non-zero vector
$\bfy\in\cc^3$ then the set of $\bfz$ solving this equation is the
intersection of an $S^5$ with a real hyperplane. The corresponding
$\cp^3$ is the twistor space (with fiber $\cp^1$) of this $S^4$
consisting of a collection of points equidistant from the chosen
origin in $S^6$.}
\end{remark}

\subsection{The conformal group}
\label{conformalgroup}
We continue to use the coordinates (\ref{coordcp7}) on $\cp^7$, and
consider first the automorphism group of the quadric $\Q^6$, as
defined by (\ref{dotproduct}). Up to a finite ambiguity, this is
isomorphic to the matrix group $\SO(8,\cc)$, but is defined in our
context by
\begin{align}
\label{autQ}
\left\{X=\mat{A&B\\C&D} \>:\> X\tp\mat{0&I\\I&0}X
= \mat{0&I\\I&0}\right\},
\end{align}
which amounts to asserting that 
\begin{align}\label{antisym}
A\tp C+C\tp A=0,\qquad B\tp D+D\tp B=0,
\end{align}
and 
\begin{align}\label{ADCB}
A\tp D+C\tp B=I.
\end{align}
We shall use these relations shortly.

The double cover of the orientation-preserving conformal group, 
$\Spin_\circ(1,7)$ (identity component), turns out to be 
exactly the subgroup
of (\ref{autQ}) consisting of matrices that preserve the twistor
fibration. Fix a matrix $M$ in this subgroup, built up from the
$4\times4$ blocks $A,B,C,D$. Given $M=\Mz\in\sE$, we therefore require
that there exists a corresponding $M'=M(\bfz')$ with the following
property. For each $\bxi\ne0$, there exists $\bxi'$ such that
\begin{align}\label{xZx}
\mat{A&B\\C&D}\col{\bxi\\M\bxi}=\col{\bxi'\\M'\bxi'}.
\end{align}
It follows that 
\begin{align}
\label{ABZ}
C+DM=M'(A+BM),
\end{align}
provided that $A+BM\ne0$.

We now list some special subgroups.
\begin{enumerate}

\item[(i)] If we take $B=C=0$ then $A\tp D=I$, and (\ref{ABZ}) implies
  that
\begin{align*}
M\in\sE\ \Rightarrow\ DMD\tp\in\sE.
\end{align*}
In particular, we can take $D=\overline A\in \SU(4)$, in which case
\begin{align}\label{AbarA}
X=\left(
\begin{matrix} 
A  & 0  \\
0  &  \overline A \\
\end{matrix}
\right).
\end{align}

\item[(ii)] Again with $B=C=0$, we can take $D=rI$ with $r\in \rr^+$,
so that
\begin{align}\label{rI}
X = \left(
\begin{matrix} 
r^{-1} I  & 0  \\
0  &   r I \\
\end{matrix}
\right).
\end{align} 

\item[(iii)] Now suppose that $B=0$ and $A=I$. Then $D=I$ and 
\begin{align}\label{IC}
X =\left(
\begin{matrix}
I  & 0 \\
C & I \\
\end{matrix}
\right),
\end{align}
with $C \in\sE$.

\item[(iv)] Consider a special case in which $A=0=D$, namely
\begin{align}\label{0II0}
X = \left(
\begin{matrix} 
0  &  I \\
I  &  0  \\
\end{matrix}
\right).
\end{align}
This is the symmetric matrix that defines the quadric itself.
\end{enumerate}
\smallbreak

We bring these example together with

\begin{proposition}
\label{confgroup}
  The orientation-preserving conformal group $\SO_{\circ}(1,7)$ is
  generated by matrices from the previous four cases. Moreover,
\begin{enumerate}
\item[(i)] corresponds to the
  group $\SO(6)$ of rotations fixing $0$ and $\infty$,
\item[(ii)]
  arises from the scaling $(z^1, z^2, z^3)\mapsto (r^2 z^1, r^2 z^2,
  r^2 z^3)$,
\item[(iii)] corresponds to the translation $( z^1, z^2,
  z^3)\mapsto(z^1-c^1, z^2-c^2, z^3-c^3)$,
\item[(iv)] arises from
  inversion in the unit sphere and minus conjugation:
\begin{align}
\label{invmap}  
(z^1,z^2,z^3)\rightarrow -\|\bfz\|^{-2}(\bz^1, \bz^2, \bz^3). 
\end{align}
\end{enumerate}
\end{proposition}

\begin{proof}
Take a point $[\bxi]\in\cp^3$ giving rise to the linear OCS on
$\rr^6$ whose $(1,0)$ forms are spanned by (\ref{etai6}) and
(\ref{etaiii}). These four equations translate into the formula
\[ \bet=M(d\bfz)\kern2pt\bxi\]
for the vector-valued 1-form $\bet$. 
To find the action of a conformal map $\bfz\mapsto \bfz'=T\circ\bfz$
we therefore need to compute $M(d\bfz')=M(T^*d\bfz)$ and define
$\bxi'$ accordingly. The induced action
\begin{align}\label{xiz}
   ([\bxi],\,\bfz)\mapsto([\bxi'],\,\bfz')
\end{align}
on twistor space can then be converted into a matrix relative to the
coordinate system (\ref{coordcp7}) using the diffeomorphism $\FFF$.
Roughly speaking, the latter replaces $\bfz$ in (\ref{xiz}) by
$\bfW=\Mz\bxi$, so that
\begin{align*}
(\FFF \circ T \circ \FFF^{-1})[\bxi,\,\bfW]
= \FFF([\bxi'],\,\bfz')
= [ \bxi',\,\bfW],
\end{align*}
where $\bfW'=M(\bfz')\bxi'$.

To tackle case (i), we work backwards. Let $A=\overline D\in \SU(4)$,
and define $\bfz'$ by
\begin{align}\label{DMD}
   M(\bfz')=D\Mz D\tp.
\end{align}
It follows from (\ref{detM}) that $\bfz\mapsto\bfz'$ is an
\emph{orthogonal} transformation of $\rr^6$, and it must lie in
$\SO(6)$ since $\SU(4)$ is connected. In this way, (\ref{DMD}) neatly
expresses the double covering \[\SU(4)\cong\Spin(6)\to
\SO(6).\] It is appropriate here to set $\bxi'=A\bxi$, so that
\[ \bet'=M(d\bfz')\bxi'=D\Mz\bxi=D\bet,\]
ensuring that the new 1-forms are linear combinations of the old ones.
Then
\begin{align*}
(\FFF \circ T \circ \FFF^{-1})[\bxi,\,\bfW] 
= [A\bxi,\,M(\bfz')A\bxi]
= [A\bxi,\,\overline A\bfW],
\end{align*}
and the matrix is (\ref{AbarA}).

Next consider the dilation $z^i\mapsto r^2 z^i$, for $r\in\rr_+$.
The 1-forms simply scale, so we may take $\bet'=\bet$.
Thus,
\begin{align*}
(\FFF \circ T \circ \FFF^{-1})[\bxi,\,\bfW]
= \FFF([\bxi],\,r^2\bfz)
= [ \bxi,\, r^2\bxi]
= [r^{-1}\bxi,\,r\bfW],
\end{align*} 
and after the projective scaling, we recover the matrix (\ref{rI})
in $\SO(8,\cc)$.

For translations, consider $z^i\mapsto z^i-c^i$ for $c^i \in\cc$.
Clearly the 1-forms are unchanged, so $\bxi'=\bxi$ whereas
\begin{align}
\label{transl}
   \bfz'=\bfz-\mathbf{c}
\end{align}
withe $\mathbf{c}=(c^1,c^2,c^3)$. Moreover,
\begin{align*}
(\FFF \circ T \circ \FFF^{-1})[\bxi,\,\bfW] 
= [\bxi,\,M(\bfz-\mathbf{c})\bxi] 
= [\bxi,\, \bfW - C\bxi],
\end{align*}
where $C\in\sE$. Thus, the lift is given by (\ref{IC}) with $-C$ in
place of $C$.

Inversion is defined by the mapping
 \begin{align*}  
   \bfz' = -\|\bfz\|^{-2}\overline\bfz.
\end{align*}
This is faithfully reflected (up to sign) by the inverse of
the matrix (\ref{Mmat}), since
\[ \Mz^{-1} = -\|\bfz\|^{-2}M(\overline\bfz) = M(\bfz'),\]
using the various properties of $\sE$. We may now define
$\bxi'=M(\bfz)\bxi=\bfW$ so that
\begin{align*}
(\FFF \circ T \circ \FFF^{-1})[\bxi,\,\bfW]
= \FFF([\bfW],\,\bfz']
= [\bfW,\,\Mz^{-1}\bfW]
= [\bfW,\,\bxi],
\end{align*}
and we obtain (\ref{0II0}).

It is well-known that the conformal group is generated by rotations,
dilations, translations and inversions \cite{SchoenYau}. 
Thus, the proof of the proposition is complete.\end{proof}

\section{Threefolds in the 6-quadric}
\label{finalsec}

We assume $J$ is an OCS defined on $S^6 \setminus K$, where $K$ is a
finite non-empty set of points. By the integrability assumption on
$J$, the graph $J(S^6 \setminus K)$ is a \emph{complex}
$3$-dimensional submanifold of $Q^6$. Moreover, it is a submanifold of
$\cp^7$ minus a subvariety consisting of finitely many vertical
$\cp^3$-s. We may therefore apply Bishop's theorem to conclude that if
the graph of $J$ has finite area then the closure is also a variety
\cite{Bishop}. The finite area condition is that
\begin{align}
\label{h62i}
\mathcal{H}^{6}( J(\rr^6 \setminus K)) < \infty,
\end{align}
where $\mathcal{H}^6$ denotes real $6$-dimensional Hausdorff measure. 

We shall see that (\ref{h62i}) is in fact implied by the finite energy
assumption (\ref{finiteenergy}). In order to compute the latter, we
endow $\Q^6$ with a ``twistor'' metric as follows. In accordance with
(\ref{VH}), the tangent space
\begin{align}
\label{VQH}
T_q(\Q^6)=V_q\oplus H_q
\end{align}
splits into the tangent space of the fiber over $z = \pi(q)$ and the
horizontal subspace determined by the Levi-Civita connection. Give
$V_q$ the Fubini--Study metric of the fiber $\cp^3$, and $H_q$ the
round metric from $S^6$. It is known that, for an appropriate choice
of scaling, this is exactly the Hermitian symmetric metric on $\Q^6$
\cite{Barbosa}. We will use this metric on $\Q^6$ to compute areas,
though any Riemannian metric constructed in a similar way would
suffice to prove the next result.

\begin{proposition}
  Let $J$ be an OCS on $S^6 \setminus K$, where $K$ is a finite set of
  points. If $J$ satisfies {\rm(\ref{finiteenergy})}, then
  {\rm(\ref{h62i})} is satisfied.
\end{proposition}

\begin{proof}
  Fix $x\in S^6$, and set $q=J(x)\in Q^6$. Let $(J_*)_x$ denote the
  differential of the smooth mapping $J\colon S^6\setminus K\to Q^6$
  at $x$. When we identify $H_q$ with $T_xS^6$ and $V_q$ with a
  subspace of $\Ext^2T_xS^6$ (as in (\ref{L20})), we may write
\[ (J_*)_x(v)=(\nabla_vJ,\>v),\qquad v\in T_xS^6.\] This is because, by its
very definition, the vertical component of (\ref{VQH}) can be
identified with the covariant derivative of $J$. The linear mapping
$(J_*)_x$ is now represented by a $12\times6$ matrix of the form
\[   D_x =
\left(
\begin{matrix}
\nabla J\\
I\\
\end{matrix}
\right),
\]
where $I$ is the $6\times6$ identity.

If we set $\Omega=S^6\setminus K$, it follows that
\[
\mathcal{H}^{6}( J(\Omega))
= \int_{\Omega}\!\sqrt{\det{D\tp D}}\>dx.
\]
This is essentially a version of the area formula
\cite{Federer,EvansGariepy}, also familiar from the classical theory
of surfaces in $\rr^n$ in which the role of $D$ is played by the
$n\times2$ matrix whose columns are the partial derivatives
$\mathbf{x}_u,\mathbf{x}_v$ and $\det(D\tp D)=EG-F^2$.

The present result now follows from the estimate
\[ 
\det(D\tp D) = \det\left(I + (\nabla J)\tp(\nabla J)\right)
\le c(1+ \|\nabla J\|^{12}),
\]
where $c$ is a universal constant. At each point $x\in S^6$, the
transpose $(\nabla J)\tp$ can be interpreted as the adjoint of a
linear transformation $T_xS^6\to\Ext^2T_xS^6$ with respect to the
natural inner products on these spaces, which are also used to define
the norm $\|\nabla J\|$. This interpretation is consistent with the
choice of Riemannian metric on $Q^6$ defined via (\ref{VQH}).
\end{proof}

In attempting to prove Theorem~\ref{t2}, we may now assume that the
closure of the graph of $J$, namely $\overline{J(\rr^6\setminus K)}$,
is an analytic variety. By Chow's Theorem, we may further assume that
it is an \emph{algebraic} subvariety of the quadric $\Q^6$
\cite{Chow}; it has complex dimension 3 and bidegree $(1,p)$, for some
integer $p\ge0$.

Now suppose that $\X$ is an arbitrary algebraic threefold (an
algebraic 3-dimensional subvariety) of $Q^6$ of bidegree $(1,p)$. We
will call a twistor fiber \emph{exceptional} if the intersection of
the fiber with $\X$ consists of more than one point.
\begin{remark}
\label{Ful} 
{\em
Given a twistor fiber $F$, suppose that the intersection $X\cap F$ is
a finite number of points. Then we have
\begin{align*}
1 = X \cdot F = \sum_{z \in X \cap F} {\rm{mult}}_z (X,F).
\end{align*}
This implies that $X\cap F$ must consist of exactly one point $z$.
Furthermore, $X$ must be smooth at $z$ and $X$ intersects $F$
transversely at $z$ \cite[Proposition 8.2 (a),(c)]{Fulton}.}
\end{remark}
A point $p \in S^6$ will be called \emph{exceptional} if the
fiber over $p$ is exceptional. It now remains to prove the
  following result that is a re-statement of Theorem~\ref{texc}.

\begin{theorem}
\label{exct}
The space of exceptional fibers of $\X \rightarrow S^6$ has real
dimension at least two, unless $\X$ is conformally equivalent to the
closure of the graph of a \wps defined on $\rr^6$.
\end{theorem}
\begin{remark}{\em
  The space of exceptional points in $S^6$, call it $E$, has the
  structure of a real algebraic variety, and may have several
  components. Since an algebraic variety has integral dimension, the
  above theorem can be rephrased to say that there are no algebraic
  examples with $E$ of dimension one, and if $E$ has dimension zero,
  it must be a single point, and moreover in this case $\X$
  corresponds to a \wp. It is clear that $\X$ will only define an OCS
  away from the set of exceptional points. This is because $J$ cannot
  be defined continuously in a neighborhood of an exceptional point
  since there are at least two directions which have a different
  limit.}
\end{remark}

\subsection{Explicit examples and the classification.}

We shall prove Theorem~\ref{exct} by applying the main result of the
paper \cite{BV}. In order for our notation to be consistent with that,
we fix attention on the non-degenerate quadric
\begin{equation}
\label{q6}
Q^6=\{[x_1,\ldots,x_8]\in\cp^7:x_1x_8-x_2x_7+ x_3x_6-x_4x_5=0\}.
\end{equation}
In identifying $Q^6$ with the twistor space of $S^6$, we shall
  further suppose that the linear subspace
\begin{align}
\label{P0}
   P_0=\{[x_1,x_2,x_3,x_4,0,0,0,0]\}\cong\cp^3
\end{align}
of $Q^6$ is \emph{vertical} in the sense of Subsection~\ref{linears}.
Only later will we be able to relate the homogeneous coordinates $x^i$
more explicitly to those in the previous section.

A key role will be played by the singular quadric $\Q^4_s$ given by
the intersection of \eqref{q6} and $\cp^5$ defined by $x_7=x_8=0$.
Thus 
\begin{align}
\label{Q4s}
\Q^4_s=\{[x_1,\ldots,x_6,0,0]\in\cp^7:x_3x_6=x_4x_5\}.
\end{align}
is defined by a quadratic form of rank 4. We may regard $Q^4_s$ as the
union of the subspaces
\begin{align}
\label{Plambda}
P_\lambda = \{ [u_0,u_1,au_2,au_3,bu_2,bu_3,0,0]\}\cong\cp^3,
\end{align}
as $\lambda=b/a$ ranges over $\cp^1$. Their common intersection is
the line
\begin{align}
\label{singL}
L=\{[u_0,u_1,0,0,0,0,0,0]\}\cong\cp^1
\end{align}
on which $Q^4_s=\Ann(L)$ is singular. 
Since the family
(\ref{Plambda}) includes the space (\ref{P0}), it follows that every
$P_\lambda$ is a vertical $\cp^3$.

Recalling \eqref{homology}, we also need to mention a special
irreducible threefold of bidegree $(1,3)$, constructed as a cone over
the Veronese embedding of $\cp^2$ in the Grassmannian
$\mathbb{G}\mathrm{r}(2,4)$. The latter is identified with the
\emph{smooth} 4-quadric $\Q^6\cap\{x_1 = x_7 = 0\}$ via the Pl\"ucker
embedding, and the threefold is the image of the weighted projective
space with homogeneous coordinates $u_0,u_1,u_2,u_3$ and weights
$(1,1,1,2)$ under the map given by
\begin{align}
  \label{coneV} 
[u_0,u_1,u_2,u_3]\mapsto[u_3,\,u_0^2,\,u_0 u_1,\,u_0
  u_2,\, u_1^2\!-\!u_0 u_2,\,u_1 u_2,\,u_2^2,\,0].
\end{align}
It features in the main classification result from \cite{BV}:

\begin{theorem}[{\cite[Theorem 2.7]{BV}}]
\label{main}
Every irreducible threefold $\X$ of bidegree $(1,p)$ in $\Q^6$ is
given by one of the following, up to the action of $\mathrm{Aut}
(\Q^6) = \mathrm{PSO}(8,\cc)$:
\begin{enumerate}
\item[(i)] $p=0$ and $\X$ is a
horizontal $\cp^3$,
\item[(ii)] $p=1$ and $\X$ is a smooth quadric
in $\cp^4\subset \cp^7$,
\item[(iii)] $p=3$ and $\X$ is the cone
over the Veronese surface given by {\rm\eqref{coneV}},
\item[(iv)] $p\ge1$ and $\X$ is a Weil divisor in the quadric
  {\rm(\ref{Q4s})}.
\end{enumerate}
\end{theorem}

\begin{remark}{\em
Different notions of divisor crop up in the study of the singular
spaces that play an essential role in our twistor theory. A smooth
example of (iv) arises as follows. The zero set in $Q^4_s$ of the
irreducible polynomial $x_1x_6-x_2x_4$ is a union $P_0\cup D$ where
$P_0$ is defined (as in (\ref{Plambda})) by $x_3=x_4=0$ and
\begin{align*}
D=\{[au_1,bu_1,au_2,au_3,bu_2,bu_3,0,0]\}\cong\cp^1\times\cp^2,
\end{align*}
is the image of the Segre embedding with $[a,b]\in\cp^1$ and
$[u_1,u_2,u_3]\in\cp^3$. In fact, every Weil divisor of degree $(1,p)$
on $\Q^4_s$ can be written as the divisor of a polynomial
$f(x_1,\ldots, x_6)$ of degree $p$ minus $(p-1)P_0$ \cite{BV}. This
example is generalized by the next proposition.}
\end{remark}

In \cite[Section 4]{BV}, it is shown that in case (iv), we have that
\begin{align}
\label{qp}
Q_\lambda = \overline{\X\cap(P_\lambda\setminus L)}
\end{align}
is a $\cp^2$ contained in $\X$. We also need to know

\begin{proposition}[{\cite[Proposition 4.4]{BV}}]
\label{4.4} 
If $\X$ is in case (iv) of Theorem \ref{main}, then $\X$ contains $L$
and is equal to the union of the $Q_\lambda$ over all $\lambda\in
\cp^1$. This is almost a disjoint union, in the sense that the only
intersections occur at points of $L$.  Furthermore, $\X$ is smooth
away from $L$.
\end{proposition}

\begin{remark}
\label{dcone}
{\em A special sub-case of case (iv) is when every $Q_{\lambda}$ 
contains $L$. In this case, $\X$ will be a {\em{double cone}} over 
a $(1,p)$ curve $C$ in a smooth $2$-quadric. That is, 
$X$ consists of all lines through points on $C$ 
and points of $L$. This case will play a crucial role in the 
proof of our main theorem. }
\end{remark}

We emphasize that the classification in Theorem \ref{main} is modulo
the action of the automorphism group $\mathrm{PSO}(8,\cc)$ of
$Q^6$. To prove our main theorem, we need to take into account special
properties of the twistor fibration and the conformal group studied in
Section~\ref{twistorprojection}. Another imported result gives a
strong restriction on the intersection of $\X$ with a twistor
fiber:
\begin{lemma}[{\cite[Proposition 3.6]{BV}}] 
\label{fibint}
If the intersection of $\X$ with a vertical $\cp^3$ has a component of
complex dimension $2$ or more, than this component is unique and is a $\cp^2$.
\end{lemma}
\noindent
With these tools, we can now prove our main result.
\begin{proof}[Proof (of Theorems~\ref{exct} and \ref{t2})]
  Suppose that we have $\X$ such that $\X\to S^6$ has at most
  a real dimension one set of exceptional fibers. Then $\X$ must
  intersect at least one of the twistor fibers in a set of complex
  dimension at least two, else $\X$ and $S^6$ are isomorphic in real
  codimension $2$. But the square of the hyperplane class is non-zero
  on $\X$, and $H^4(S^6,\mathbb{Z})=0$, so this is not possible.

We identify $\rr^6$ with $S^6 \setminus \{\infty\}$. By performing 
a conformal transformation, we may assume without loss of generality 
that $\X$ intersects the fiber $F_{\infty} \cong\cp^3$ over infinity 
in a set of complex dimension $2$. By Lemma \ref{fibint}, this 
intersection consists of a $\cp^2$
and (perhaps) some lower dimensional components. We will now invoke
Theorem~\ref{main}.\smallbreak

If $p=0$ and $\X$ is a horizontal $\cp^3$, then this corresponds to a
constant OCS. Indeed, $\X$ hits $F_{\infty}$ in a $\cp^2$, and hits every
other fiber in a single point. Using $\SU(4)$, we can 
arrange so that this $\cp^2$ is the same as that corresponding 
to $\bJ$ in \eqref{1cp}. Since there is a unique horizontal 
$\cp^3 \subset Q^6$ containing this $\cp^2$, this proves 
that $J$ is conformally equivalent to $\bJ$. \smallbreak

If $p=1$ and we have a smooth quadric $\Q^3$, then $\X\cap F_{\infty}$ 
cannot be a $\cp^2$ because the maximal linear subspace contained in $\Q^3$
is a $\cp^1$.\smallbreak

If $p=3$ and we have a cone over the Veronese surface, then this $\X$ does not
contain any $\cp^2$-s either. To see this, the image of a $\cp^2$
in the Veronese surface would be a linear subspace of complex dimension one or two.
Of course, it is not a $\cp^2$, since the Veronese image is not linear. It is
not a $\cp^1$ either, because all curves in the Veronese surface have even
degree. Indeed, there are no lines inside the Veronese surface in
$\cp^5$.\smallbreak

It remains to consider the case (iv), in which $\X$ lies in the image
$\tilde{Q}^4_s$ of the quadric \eqref{Q4s} under a projective
transformation of $Q^6$. 
We let $\tl$ denote the singular line of  $\tilde Q^4_s$, 
and we continue to denote the fiber $\cp^3$ at infinity by $F_\infty$.

\begin{claim} The line $\tl$ must be contained in $F_{\infty}$.
\end{claim}
\begin{proof}
First, we rule out that $\tl$ sits completely in some fiber $F$ disjoint
from $F_{\infty}$. To see this, if it sits completely in some other
fiber $F$ then the span of $\tl$ and the $\PP^2 \subset F_{\infty} \cap
X$ would be an isotropic $\PP^4$.  Indeed, $\tl$ and $\PP^2$ are both
isotropic and orthogonal to each other, because $\tilde{Q}^4_s
=\Ann(\tl) \cap Q^6 \subset\Ann(\tl)$.

Second, if $\tl$ has a transversal intersection point with a 
fiber $F\cong\cp^3$ different from $F_{\infty}$ this forces $F$ to be an 
exceptional fiber. 
Suppose, on the contrary, that $F$ is not exceptional, and 
that $z\in F\cap \tl$ is a transversal 
intersection point and that $z$ is the unique point of $F \cap X$.  
Now fix
$\lambda\in\cp^1$. 
Since both $F$ and $P_\lambda$ are vertical, the
intersection $F\cap P_\lambda$ must be a $\cp^1$. Indeed, two 
vertical $\cp^3$-s are either disjoint, intersect in a $\cp^1$, 
or are the same. 
Of course, $F$ intersects each $P_\lambda$ in the
point $z$ (since each $P_{\lambda}$ contains $\tl$, and $z \in \tl$),
so the first possibility does not occur. 
If $F = P_{\lambda}$, then $F$ would contain
$Q_{\lambda}$, and would thus be exceptional,
contrary to assumption. 
Therefore $F\cap P_\lambda = \cp^1$. Also, $X \cap P_{\lambda}$ 
contains $Q_{\lambda} = \cp^2$. Since any $\cp^1 \subset P_{\lambda}$
and any $\cp^2 \subset P_{\lambda} = \cp^3$ must intersect, 
$F \cap X$ must therefore contain a point of $Q_{\lambda}$. 
By assumption, this point must be $z$. This shows that $z$ is 
contained in {\em{every}} $Q_{\lambda}$.
Proposition~\ref{4.4} now implies that $\X$ is a
cone with vertex $z$, so it is singular at $z$, since $\deg(\X)>1$.
Since it is singular at $z$, the intersection of $\X$ with $F$ at $z$
cannot be transversal (see Remark \ref{Ful}), which is a contradiction.

The previous paragraph implies that $\tl$ cannot have a transversal 
intersection point with {\em{any}} fiber. Indeed, since different points 
of $\tl$ belong to different $F$-s, we would have a real dimension $2$ set 
of exceptional fibers, which contradicts the assumption on the dimension of 
the exceptional set. Together with the first paragraph, this proves 
the claim. 
\end{proof}

We may now assume that the fiber $F_{\infty}$ contains $\tl$. Under
this assumption we have the following. 
\begin{claim}
\label{cl2}
For any twistor fiber $F$ other than $F_{\infty}$, $\tilde{Q}^4_s \cap F$ is a
$\cp^1$.
\end{claim}
\begin{proof}
To see this, recall that 
$\tilde{Q}^4_s = Q^6 \cap\Ann(\tl)$, so $\tilde{Q}^4_s \cap F =\Ann(\tl) \cap F$. 
This space is orthogonal to $\tl$, so if this were a 
$\cp^3$ or $\cp^2$, we would have too large an isotropic subspace. 
Finally, $\dim\Ann(\tl) = 5$ and $\dim(F) = 3$, so the 
intersection must be at least a line. 
\end{proof}
Next, if the
projective plane $Q_\lambda$ does not contain $\tl$ then it cannot lie
entirely in a fiber and by Lemma \ref{cp2} and 
Claim \ref{cl2}, there is a unique exceptional fiber 
$F_\lambda$ for which
\begin{align}
\label{Fl}
F_{\lambda} \cap Q_\lambda = F_\lambda\cap \tilde{Q}_s^4\cong\cp^1.
\end{align}
Next, we notice that, among those $\lambda$ for which 
$Q_{\lambda}$ does not contain $\tl$, 
the correspondence $\lambda \mapsto F_{\lambda}$
is injective. To see this, if $F_{\lambda_1} = F_{\lambda_2} = F$,
then \eqref{Fl} implies that $F \cap Q_{\lambda_1} = F \cap Q_{\lambda_2}$
is the same $\cp^1$. Since distinct $Q_{\lambda}$-s
can only intersect at points of $\tl$ (Proposition \ref{4.4}), 
we must have $Q_{\lambda_1}=Q_{\lambda_2}$,
and thus  $\lambda_1 = \lambda_2$. Consider the subset of $\cp^1$ 
given by the $\lambda$-s for which $Q_{\lambda}$ contains
$\tl$. This is an algebraic set, so it either consists of a 
finite number of points, or is the entire $\cp^1$. 
If it is finite, then there is a real $2$-dimensional set 
of exceptional fibers, which is contrary to assumption. 
So it must be the entire $\cp^1$, and $X$ is a double cone
(see Remark \ref{dcone}). 

\begin{claim}
If $X$ is a double cone over $\tl\subset F_\infty$, then for any
other twistor fiber $F$, $X\cap F$ is a point. 
\end{claim}
\begin{proof}
  Otherwise, $X \cap F$ would have to be at least one-dimensional and
  thus would equal the whole $\cp^1=\tilde{Q}_s^4\cap F$ from Claim
  \ref{cl2}. In this case, $X$ would then contain the span of $\tl$
  and this $\cp^1$, which is an isotropic $\cp^3$. This contradicts
  the fact that $X$ is irreducible and $p>0$.
\end{proof}
This claim implies that $X$ is a graph over $\rr^6$, so yields an
globally defined OCS on $\rr^6$.  To see that it is a \wp, we argue
as follows.  From Proposition \ref{confgroup}, ${\rm{SO}}(6)$ lifts to
an action of $\SU(4)$ on $F_{\infty}$.  Since $\U(4)$ acts freely and
transitively on unitary bases of $\cc^4$, it follows that $\SU(4)$
acts transitively on full flags.  Furthermore, $-I \in \SU(4)$ only
changes all signs of the basis elements, so $\SO(6)$ also acts
transitively on full flags.  Consequently, we may assume, after
reverting to our previous coordinates
\begin{align*}
x_1=\xi_0,\ x_2=\xi_{12},\ x_3=\xi_{13},\ x_4=\xi_{23},\
x_5=W_1,\ x_6=W_2,\ x_7=W_3,\ x_8=W_{123},
 \end{align*}
that 
\begin{align*}
\X\cap F_\infty=\{[ 0,0,0,0, W_1, W_2, W_3, 0]\}\cong\cp^2, 
\end{align*}
and that $\tl = L$ is given by $W_3=0$ (and thus $\tilde{Q}^4_s =
Q^4_s$). From \cite[Lemma 3.5]{BV}, $\X$ must lie in the annihilator
of $L$, which forces $\xi_{13} = \xi_{23} = 0$. This implies that $\X$
corresponds to a \wp\ OCS described by \eqref{WPS6}.
This completes the proof of Theorem~\ref{exct} and Theorem~\ref{t2} (i). 

In twistor coordinates, we can describe $\X$ as a graph over 
$\rr^6 = \cc^3$ by
\begin{align}\label{xixi} 
 [\xi_0,\>\xi_{12},\>0,\>0,\>\xi_0 z^1-\xi_{12}\bz^2,
  \>\xi_0 z^2 +\xi_{12}\bz^1,\>\xi_0 z^3,\>\xi_{12} z^3],
\end{align}
where the meromorphic function $f:\cc\to\cp^1$ is given by $f(z^3) = 
[\xi_0(z^3), \xi_{12}(z^3)]$.
From this expression, one sees again
directly that $\X$ is a double cone over a $(1,p)$-curve $\Sigma
\subset \cp^1\times \cp^1$. This $(1,p)$-curve is given by the
closure in the smooth $2$-quadric of the graph
\begin{align}
\label{gq2} 
[\xi_0(z^3),\>\xi_{12}(z^3),\>0,\>0,\>0,\>0,
\>\xi_0(z^3) z^3,\>\xi_{12}(z^3) z^3]
\end{align}
as $z^3\to\infty$. Note that since $X$ is algebraic, up to a constant multiple
$\xi_0$ and $\xi_{12}$ can be chosen as polynomial functions of $z_3$
which do not vanish simultaneously. We then clearly have
\begin{align*}
\deg(f) = \max \{  \deg (\xi_0), \deg (\xi_{12}) \} = p = \deg_{\cp^7}X - 1. 
\end{align*}
Theorem~\ref{t2} (ii) follows easily since 
the action of the conformal group is linear (so preserves the degree), and
$\deg_{\cp^7} X = 1$ if and only if $f$ is a constant if and only if
$J$ is conformally equivalent to the standard OCS $\bJ$ on $\rr^6$
(from work in Sections~\ref{warped} and \ref{finding}). 
Theorem~\ref{t2} (iii) follows from Proposition~\ref{cosymp} above 
and Corollary~\ref{habi} below.
\end{proof}
\subsection{Final Remarks.}
We conclude this section with a few observations. The OCSes 
under consideration, while not conformally 
equivalent to $\cc^3$ for $p > 0$, are in fact
{\em{biholomorphic}} to $\cc^3$.
\begin{corollary}
\label{habi}
  Let $J$ be an OCS on $\rr^6$ with finite energy. Then there exists a
  harmonic biholomorphism $F\colon(\rr^6,J)\to (\cc^3, \bJ)$.
\end{corollary}

\begin{proof} As seen above, finite energy implies that 
$f$ is algebraic, so we can assume in
  \eqref{xixi} that $\xi_0,\xi_{12}$ are polynomials in $z^3$ that do
  not vanish simultaneously. We specify the representative of the
  projective class by supposing that $\xi_0$ has a fixed non-zero
  value at some point. We claim that the three functions
\begin{align*}
\{W_1=\xi_0(z^3)z^1-\xi_{12}(z^3)\bz^2, \
W_2=\xi_0(z^3)z^2 +\xi_{12}(z^3)\bz^1, \
z^3\}
\end{align*}
are the components of a holomorphic mapping
$F\colon(\rr^6,J)\to(\cc^3, \bJ)$. This is best seen directly from the
definition \eqref{wpo} of a \wp\ OCS $J$. Since $\xi_0,\xi_{12}$ are
themslves holomorphic functions, we may express the space of $(1,0)$
forms for $J$ as
\begin{align*}
\left<\eta_1,\,\eta_2,\,dz^3\right>=\left<dW_1,\,dW_2,\,dz^3\right>
\end{align*}
in the notation of \eqref{etaeta}. 

It remains to check that $F$ is bijective, but this is true because
$(z^1,z^2, z^3)$ can be recovered from $(W^1,W^2,z^3)$ using the
equations
\begin{align*}
\begin{split}
(|\xi_0|^2+|\xi_{12}|^2)z^1 = \bar\xi_0 W_1+\xi_{12}\bar W_2,\\
(|\xi_0|^2+|\xi_{12}|^2)z^2 = \bar\xi_0 W_2-\xi_{12}\bar W_1,
\end{split}
\end{align*}
that are special cases of \eqref{zzz}.

The adjective ``harmonic'' refers to the Euclidean metric and the
fact that \begin{align*} 
\Delta F = 4\sum_{i=1}^n\frac{\pd^2F}{\pd z^i\pd\bz^i}
\end{align*}
is zero because the variables are sufficiently separated. The result
is also a consequence of Proposition~\ref{cosymp}(ii) and the fact
that any holomorphic map from a cosymplectic manifold to $\cc$ is
necessarily harmonic \cite{Lic70,Salamon1985}.
\end{proof}

\begin{remark}{\em
  It is also true that the variety $X$ is \emph{rational}, that is,
  $X$ is birational to $\cp^3$, though it is not biholomorphic to
  $\cp^3$ unless $f$ is constant. Moreover, in higher dimensions, it
  is true that the closure of the graph of any algebraic OCS $J$
  globally defined on $\rr^{2n}$ is rational, but we omit the
  proof. However, we do not know whether $(\rr^{2n},J)$ is necessarily
  biholomorphic to $(\cc^n, \bJ)$ in higher dimensions.}
\end{remark}

By taking the meromorphic function $f$ to be non-algebraic, we obtain
the examples mentioned in the Introduction. In this case, the closure
of (\ref{gq2}) in $\cp^1 \times \cp^1$ will contain the entire missing
$\cp^1$ (except for possibly a finite set of points), so its preimage
in $\Q^6$ will essentially contain the whole twistor fiber over
infinity. By Bishop's Theorem, these examples will necessarily have
infinite energy.

Finally, we return to the examples on $T^6$ from Subsection
\ref{warptori}. Such a structure will necessarily lift to a \wps 
$\tilde{J}_6$ on $\rr^6$ with $f$ a doubly-periodic
meromorphic function on $\cc$, invariant under a lattice, which is
equivalent to a holomorphic function $f : (T^2, J_2) \rightarrow
\cp^1$.  Since $f$ is necessarily non-algebraic if non-constant,
$\tilde{J}_6$ will also necessarily have infinite energy.

\bibliography{BSVref}

\enddocument